\documentclass[11pt]{article}

\usepackage{amssymb}
\usepackage{amsmath,amsfonts,mathtools}
\usepackage{amsthm}
\usepackage{latexsym}
\usepackage{mathrsfs}
\usepackage{color}
\usepackage{float}
\usepackage{enumitem}
\usepackage{algorithm,algorithmic}
\allowdisplaybreaks[3]
 
\usepackage{bbm}
\usepackage{microtype}
\usepackage[a4paper]{geometry}
\usepackage{pifont}
\usepackage{latexsym}

\usepackage{xcolor}
\definecolor{mywords}{named}{blue}
\definecolor{myparagraph}{named}{red}
\usepackage{booktabs}

\usepackage[a4paper]{geometry}
\newcommand{\beq}{\begin{equation}}
\newcommand{\eeq}{\end{equation}}
\newcommand{\beqa}{\begin{eqnarray}}
\newcommand{\eeqa}{\end{eqnarray}}
\newcommand{\beqas}{\begin{eqnarray*}}
\newcommand{\eeqas}{\end{eqnarray*}}
\newcommand{\ba}{\begin{array}}
\newcommand{\ea}{\end{array}}
\newcommand{\bi}{\begin{itemize}}
\newcommand{\ei}{\end{itemize}}
\newcommand{\nn}{\nonumber}

\newtheorem{lemma}{Lemma}
\newtheorem{theorem}{Theorem}
\newtheorem{corollary}{Corollary}
\newtheorem{definition}{Definition}
\newtheorem{remark}{Remark}
\newtheorem{assumption}{Assumption}

\def\bR{{\mathbb R}}
\def\cB{{\mathcal B}}
\def\cC{{\mathcal C}}

\def\cL{{\mathcal L}}

\def\cO{{\mathcal O}}
\def\cS{{\mathcal S}}
\def\cU{{\mathcal U}}
\def\cX{{\mathcal X}}
\def\cY{{\mathcal Y}}
\def\dist{{\rm dist}}
\def\dom{{\rm dom}}
\def\hx{{\hat x}}
\def\nn{{\nonumber}}
\def\rc{{\rm c}}
\def\tx{{\tilde x}}

\DeclareMathOperator{\diag}{diag}
\def\rC{{\rm C}}

\title{A first-order method for nonconvex--nonconcave minimax problems under a local Kurdyka--\L{}ojasiewicz condition}
\author{
Zhaosong Lu
\thanks{
Department of Industrial and Systems Engineering, University of Minnesota, USA (email: {\tt zhaosong@umn.edu}, {\tt wan02269@umn.edu}). This work was partially supported by the Air Force Office of Scientific Research under Award FA9550-24-1-0343, the Office of Naval Research under Award N00014-24-1-2702, and the National Science Foundation under Awards 2211491 and 2435911.
}
\and
Xiangyuan Wang
\footnotemark[1]
}
\date{July 2, 2025 (Revised: May 19, 2026)}

\begin{document}
\maketitle

\begin{abstract}
We study a class of nonconvex--nonconcave minimax problems in which the inner maximization problem satisfies a local Kurdyka--\L{}ojasiewicz (KL) condition that may vary with the outer minimization variable. In contrast to the global KL or Polyak--\L{}ojasiewicz (PL) conditions commonly assumed in the literature---which are significantly stronger and often too restrictive in practice---this local KL condition accommodates a broader range of practical scenarios. However, it also introduces new analytical challenges. In particular, as an optimization algorithm progresses toward a stationary point of the problem, the region over which the KL condition holds may shrink, resulting in a more intricate and potentially ill-conditioned landscape. To address this challenge, we show that the associated maximal function is locally generalized H\"{o}lder smooth. Leveraging this key property, we develop an inexact proximal gradient method for solving the minimax problem, where the inexact gradient of the maximal function is computed by applying a proximal gradient method to a KL-structured subproblem. Under mild assumptions, we establish complexity guarantees for computing an approximate stationary point of the minimax problem.
\end{abstract}


\medskip


\section{Introduction} \label{sec:intro}
In this paper, we consider a nonconvex--nonconcave minimax problem of the form
\begin{equation} \label{intro_problem}
\min_x \max_y \left\{ f(x, y) + p(x) - q(y) \right\},
\end{equation}
where $f$ is a smooth function that is nonconvex in $x$ and nonconcave in $y$, and $p$ and $q$ are possibly nonsmooth, closed, and simple convex functions. 

Problem~\eqref{intro_problem} arises in a wide range of applications in machine learning and operations research, including generative adversarial networks~\cite{arjovsky2017wasserstein, goodfellow2020generative}, reinforcement learning~\cite{dai2018sbeed, omidshafiei2017deep}, adversarial training~\cite{madry2017towards, sinha2017certifying}, and distributionally robust optimization~\cite{bertsimas2011theory, blanchet2025distributionally,rahimian2022frameworks}. Despite its broad applicability, problem~\eqref{intro_problem} remains computationally challenging due to its inherent nonconvex--nonconcave structure. For instance, computing a global Nash equilibrium---an important special case of~\eqref{intro_problem}---is generally NP-hard (see, e.g.,~\cite{jin2019minmax}).

In recent years, significant progress has been made under specific structural assumptions on problem~\eqref{intro_problem}. For example, several studies focus on the special case where $q=0$ and impose the global Polyak-\L ojasiewicz (PL) condition on the inner maximization problem of~\eqref{intro_problem}, which is generally weaker than the strong concavity assumption. Under this condition, gradient descent–ascent type methods have been developed, and complexity guarantees have been established for finding approximate stationary points. Remarkably, these guarantees match those obtained under the strong concavity assumption for the inner maximization subproblem of~\eqref{intro_problem} (see, e.g.,~\cite{huang2023enhanced,nouiehed2019solving, xu2023zeroth,yang2022faster}). In addition, first-order methods have been developed for problem~\eqref{intro_problem} from a variational inequality perspective, typically assuming the existence of a weak Minty variational inequality solution (see, e.g., \cite{bohm2022solving,cai2022accelerated,liu2021first,pethick2023escaping}).

More recently, \cite{li2025nonsmooth, zheng2025doubly, zheng2023universal} studied a class of minimax problems of the form~\eqref{intro_problem}, where $p$ and $q$ are indicator functions of simple convex compact sets and a global Kurdyka-\L ojasiewicz (KL) condition is imposed on the inner maximization problem. They developed gradient descent–ascent-type methods that alternately update the $x$ and $y$ variables using first-order schemes, and established complexity guarantees for finding approximate stationary points.  Notably, the class of minimax problems considered in \cite{li2025nonsmooth, zheng2025doubly, zheng2023universal} is significantly broader than those studied in \cite{huang2023enhanced, nouiehed2019solving,xu2023zeroth,yang2022faster}, since the KL condition generalizes the PL condition (which corresponds to the KL condition with exponent $1/2$) and accommodates nonsmooth objectives. 

However, requiring the KL property to hold globally is often too restrictive in practice, which limits the applicability of the proposed methods. To address this limitation, we relax the global KL assumption on the inner maximization problem of~\eqref{intro_problem} that is imposed in the existing literature. Specifically, we assume that for each fixed outer variable $x \in \dom\,p$, the KL condition holds only on a level set of the inner variable $y$. Moreover, this level set may depend on $x$, and its size may vary accordingly (see Assumption \ref{ass:Lip_Smo_KL}(iii) for details). This weaker assumption accommodates a broader range of practical scenarios but also introduces new analytical challenges. In particular, as an optimization algorithm progresses toward a stationary point of~\eqref{intro_problem}, the region over which the KL condition is valid may shrink, resulting in a more intricate and potentially ill-conditioned landscape. Moreover, we consider more general functions $p$ and $q$, beyond the indicator functions of simple convex compact sets considered in prior works  \cite{li2025nonsmooth, zheng2025doubly, zheng2023universal}, thereby further broadening the class of minimax problems under consideration.

In this paper, we study problem~\eqref{intro_problem} under the aforementioned local KL condition and other mild assumptions (see Assumption~\ref{ass:Lip_Smo_KL} below). In particular, we show that the maximal function, defined as $F^*(x) := \max_y \{ f(x, y) - q(y) \}$, is locally generalized H\"{o}lder smooth on the set $\{x \in \dom\,p: 0 \notin \partial \Psi(x) \}$, where $\Psi(x) := F^*(x) + p(x)$ is the value function of problem~\eqref{intro_problem} (see Theorem~\ref{thm:Phi_smooth}). Leveraging this key property, we develop an inexact proximal gradient method (Algorithm~\ref{algo_implement:whole}) to solve the problem $\min_x \{ F^*(x) + p(x) \}$, which is equivalent to the original minimax problem~\eqref{intro_problem}. Specifically, given the current iterate $x^k$, we apply a proximal gradient method (Algorithm~\ref{algo_implement:subsolver}) to approximately solve the subproblem $\max_y \{f(x^k, y) - q(y)\}$, starting from the previous inner iterate $y^{k-1}$, and obtain an approximate solution $y^k$. We then perform an inexact proximal gradient step to compute $x^{k+1}$, using $-\nabla_x f(x^k, y^k)$ as the forward direction together with a carefully chosen step size. We also establish complexity guarantees for the proposed method for computing an approximate stationary point of problem~\eqref{intro_problem}.

The main contributions of this paper are summarized below.

\begin{itemize}
\item  We establish a local generalized H\"older smoothness property for the maximal function $F^*$ under a local KL condition, which plays a crucial role in the development of a method for solving problem~\eqref{intro_problem} (see Theorem~\ref{thm:Phi_smooth}). 
\item 
We propose an inexact proximal gradient method for finding approximate stationary points of problem~\eqref{intro_problem}. Under mild assumptions, we establish that this method achieves an \emph{iteration complexity} of $\widetilde{\mathcal{O}}\big(\epsilon^{-\max\{(1-\theta)^{-1},\, \theta^{-1}\sigma\}}\big)$, and a \emph{first-order oracle complexity} of $\widetilde{\mathcal{O}}\big(\epsilon^{-(1-\theta)^{-1}(2\theta^2 - 2\theta + 1)\max\{(1-\theta)^{-1},\, \theta^{-1}\sigma\}}\big)$, measured by the number of gradient evaluations, for finding an 
$\mathcal{O}(\epsilon)$-approximate stationary point of \eqref{intro_problem}. Here, $\theta$ and $\sigma$ are the parameters for the local KL condition given in Assumption \ref{ass:Lip_Smo_KL}.\footnote{$\widetilde{\mathcal{O}}(\cdot)$ represents $\mathcal{O}(\cdot)$ with logarithmic factors hidden.}
\end{itemize}

The rest of this paper is organized as follows. Subsection~\ref{subsec:notation} introduces the notation, terminology, and assumptions used throughout the paper. In Section~\ref{sec:Holder}, we establish a local generalized H\"{o}lder smoothness property of the maximal function. Section~\ref{sec:subsolver} presents a proximal gradient method for minimizing functions that satisfy a local KL property. In Section~\ref{sec:FOD}, we propose an inexact proximal gradient method for solving problem~\eqref{intro_problem} and analyze its complexity. Section~\ref{sec:Numerical} presents preliminary numerical results illustrating the performance of the proposed method. In Section~\ref{sec:proof}, we provide the proof of the main results. Finally, we make some
concluding remarks in Section~\ref{sec:concluding}.

\subsection{Notation, terminology, and assumptions} \label{subsec:notation}
The following notation will be used throughout the paper. Let $\mathbb{R}^n$ denote the $n$-dimensional Euclidean space, and let $\overline{\mathbb{R}}=\mathbb{R}\cup\{\infty\}$. The standard inner product, $\ell_1$-norm, $\ell_\infty$-norm, and Euclidean norm are denoted by $\langle \cdot, \cdot \rangle$, $\|\cdot\|_1$, $\|\cdot\|_\infty$, and $\|\cdot\|$, respectively. For any two points $u, v \in \mathbb{R}^n$, the notation $[u, v]$ denotes the line segment connecting $u$ and $v$. Given a point $x$ and a closed set $S \subset \mathbb{R}^n$, let $\dist(x,S)$ denote the distance from $x$ to $S$. 
The closed ball centered at $x \in \mathbb{R}^n$ with radius $r$ is denoted by $\cB(x, r)$. In addition, $\operatorname{conv}(\cdot)$, $\operatorname{aff}(\cdot)$, $\operatorname{ri}(\cdot)$, and $\operatorname{int}(\cdot)$ denote the convex hull, affine hull, relative interior, and interior of a set, respectively.

For an extended real-valued function $\phi : \mathbb{R}^n \to \overline{\mathbb{R}}$, its \emph{domain} is denoted by $\dom\,\phi$, i.e., $\mathrm{dom}\, \phi = \{x : \phi(x) < \infty\}$. Such $\phi$ is called \emph{proper} if $\dom\,\phi \neq \emptyset$, and it is called  \emph{closed} (or \emph{lower semicontinuous}) if $\liminf_{z\to x} \phi(z)\geq \phi(x)$ holds for all $x\in\bR^n$. The \emph{limiting subdifferential} (see, e.g., \cite[Definition 8.3(b)]{rockafellar2009variational}) of a proper closed function $\phi$ at $x \in \mathrm{dom}\, \phi$ is defined as
\[
\partial \phi(x) := \left\{ v \in \mathbb{R}^n : \exists\, x^k \xrightarrow{\phi} x,\ v^k \to v\ \text{with }  \liminf_{z \to x^k,\, z \neq x^k} \frac{\phi(z) - \phi(x^k) - \langle v^k, z - x^k \rangle}{\|z - x^k\|} \geq 0 \ \,\forall k\right\}.
\]
% and the \emph{horizon subdifferential} (see, e.g., \cite[Definition 8.3(c)]{rockafellar2009variational}) of $\phi$ at $x \in \mathrm{dom}\, \phi$ is defined as
% \[
% \partial^{\infty} \phi(x) := \left\{ v \in \mathbb{R}^n : \exists\, x^k \xrightarrow{\phi} x,\ v^k \to v,\ \lambda_k\downarrow0 \ \text{with } v^k\in\partial \phi(x^k) \ \,\forall k\right\},
% \]
% where $\lambda_k\downarrow0$ means $\lambda_k>0$ and $\lambda_k\to 0$, and $x^k\overset{\phi}{\to} x$ means both $x^k\to x$ and $\phi(x^k)\to \phi(x)$. 
In addition, we use $\partial_{x_i} \phi$ to denote the limiting subdifferential of $\phi$ with respect to $x_i$. If $\phi$ is continuously differentiable, then $\partial \phi$ coincides with the gradient $\nabla \phi$. Besides, if $\phi$ is convex, then $\partial \phi$ corresponds to the classical convex subdifferential.  Moreover, if $\phi$ is proper and closed, $\partial \phi$ is outer semicontinuous on $\dom\, \phi$. That is, for any $x \in \dom\, \phi$, if $x^k\overset{\phi}{\to} x$ and $v^k \in \partial \phi(x^k)$ with $v^k \to v$, then $v \in \partial \phi(x)$ (see, e.g., \cite[Proposition 8.7]{rockafellar2009variational}).

Suppose that $\phi : \mathbb{R}^n \to \overline{\mathbb{R}}$ is a locally Lipschitz continuous function on $\cX$. The \emph{Clarke subdifferential} (see, e.g., \cite[Definition (1.1)]{clarke1975generalized}) of $\phi$ at $x \in \cX$, denoted by $\partial^\rC \phi(x)$, is defined as
\beq \label{clarke}
\partial^\rC \phi(x) := {\rm conv}\{v : \exists\, x^k \rightarrow x  \text{ such that } \nabla \phi(x^k) \to v \}.
\eeq
The \emph{restricted Clarke subdifferential} of $\phi$ with respect to $\mathcal{X}$, denoted by $\partial^\rC_\cX \phi$, is defined as
\beq \label{r-clarke}
\partial^\rC_\cX \phi(x) := {\rm conv}\{v : \exists\, x^k \in \cX \rightarrow x  \text{ such that } \nabla \phi(x^k) \to v \} \quad \forall x\in \cX.
\eeq
Clearly, $\partial^\rC_\cX \phi(x) \subseteq \partial^\rC \phi(x)$ for all $x\in \cX$.  
When $\partial^\rC_\cX \phi(x)$ is a singleton, we denote its unique element by $\nabla^\rC_\cX \phi(x)$.

A function $\phi : \mathbb{R}^n \to \overline{\mathbb{R}}$ is called \emph{$L_\phi$-Lipschitz continuous} on $\cX$ if $\lvert \phi(x)-\phi(y)\rvert\le L_\phi\|x-y\|$ for all $x,y\in  \cX$, and \emph{$L_{\nabla \phi}$-smooth} on $ \cX$ if $\|\nabla \phi(x)-\nabla \phi(y)\|\le L_{\nabla \phi}\|x-y\|$ for all $x,y\in \cX$.
In addition, we introduce the following notion of \emph{generalized H\"older smoothness}.
\begin{definition}[{\bf generalized H\"older smoothness}]
Suppose that $\phi : \mathbb{R}^n \to \overline{\mathbb{R}}$ is a locally Lipschitz continuous function on $\cX$ and that $\partial^\rC_\cX \phi(x)$ is a singleton for all $x\in \cX$. The function $\phi$ is said to be generalized H\"older smooth on $\cX$ if there exist $L_1 \geq 0$, $L_2 \geq 0$, and  $\nu \in (0, 1)$ such that 
\[
\|\nabla^\rC_\cX \phi(x)-\nabla^\rC_\cX \phi(y)\|\le L_1 \|x-y\|+L_2\|x-y\|^\nu \quad \forall x,y\in \cX.
\]  
\end{definition}

% In addition, suppose that $\phi$ is a locally Lipschitz continuous function on $\cX$ and that $\partial^\rC_\cX \phi(x)$ is a singleton for all $x\in \cX$. The function $\phi$ is said to be \emph{generalized H\"older smooth} on $\cX$ if there exist $L_1 \geq 0$, $L_2 \geq 0$, and  $\nu \in (0, 1)$ such that 
% \[
% \|\nabla^\rC_\cX \phi(x)-\nabla^\rC_\cX \phi(y)\|\le L_1 \|x-y\|+L_2\|x-y\|^\nu \quad \forall x,y\in \cX.
% \]

We now introduce an approximate stationary point for the problem $\min_x \phi(x)$, where $\phi$ is a proper closed function. A similar notion has been used in the context of weakly convex optimization (see, e.g., \cite{davis2019stochastic}). As will be shown later, under mild assumptions, the minimax problem~\eqref{intro_problem} can be viewed as a special case of this problem. Consequently, the following definition applies to problem~\eqref{intro_problem} as well.

\begin{definition}[{\bf $(r,\epsilon)$-stationary point}] \label{r-eps-stat}
Suppose $\phi$ is a proper closed function.  For any $\epsilon > 0$ and $r \geq 0$, a point $\bar{x}$ is called an $(r,\epsilon)$-stationary point of the problem $\min_x \phi(x)$ if $\bar{x}\in \dom\,\phi$ and $\dist(\bar{x}, \mathscr{S}_\epsilon) \leq r$,
where $\mathscr{S}_\epsilon = \{x: \dist(0, \partial \phi(x)) \le \epsilon\}$.
\end{definition}

It should be noted that when $\phi$ is a locally Lipschitz continuous function, any $(r,\epsilon)$-stationary point $\bar{x}$ of $\phi$ is also an $(r,\epsilon)$-Goldstein stationary point of $\phi$, that is, $\dist(0,\partial_r \phi(\bar{x})) \le \epsilon$, where
\[
\partial_r \phi(\bar{x}) := \operatorname{conv}\Big( \bigcup_{x \in \mathcal{B}(\bar{x}, r)} \partial \phi(x) \Big).
\]

We next introduce additional notation and assumptions for problem~\eqref{intro_problem}. For convenience, we define

\begin{align}
&\mathcal{X} := \dom \, p, \quad \mathcal{Y} := \dom \, q, \quad F(x, y) := f(x, y) - q(y), \label{eq:F} \\  
& F^*(x) := \max_y F(x, y), \quad  Y^*(x) := \{y : F(x, y) = F^*(x)\}, \label{eq:F-star} \\
& \Psi(x) := F^*(x) + p(x), \quad  \Psi^* := \min_x \Psi(x). \label{eq:Psi}
\end{align}

We assume that problem~\eqref{intro_problem} satisfies the following assumption.

\begin{assumption} \label{ass:Lip_Smo_KL}
\begin{enumerate}[label=\textup{(\roman*)}]
    \item For any fixed $y \in \mathcal{Y}$, the function $f(\cdot, y)$ is $L_f$-Lipschitz continuous on $\cX$. Moreover, the function $f: \mathbb{R}^n \times \mathbb{R}^m \to \mathbb{R}$ is $L_{\nabla f}$-smooth on $\cX \times \mathcal{Y}$. 
    
    \item $ p: \mathbb{R}^n \to \mathbb{R} \cup \{+\infty\} $ is proper closed convex, $ q:\mathbb{R}^m \to \mathbb{R} \cup \{+\infty\} $ is proper closed convex, and the proximal operators of $p$ and $q$ can be computed exactly. In addition, we assume that $\operatorname{aff}(\cX) = \bR^n$.\footnote{The assumption $\operatorname{aff}(\mathcal{X}) = \mathbb{R}^n$ is imposed merely for convenience. Without this assumption, the results of the paper remain valid and the analysis proceeds identically, except that the gradients and subdifferentials should be understood as being defined relative to $\operatorname{aff}(\mathcal{X})$.}
    
    \item For any fixed $x \in \cX$, $ \max_y F(x, y) $ has a nonempty solution set and a finite optimal value. The function $ F $ satisfies the following local Kurdyka-\L ojasiewicz (KL) condition in $y$: there exist constants $ C > 0 $, $\theta \in [1/2,1)$, $\gamma > 0$, and $\sigma > 0$ such that for any $x \in \cX$,
    \beq \label{KL_y}
    C(F^*(x) - F(x, y))^{\theta} \leq \dist(0, \partial_y F(x, y)) \qquad \forall y \in \mathcal{L}(x),
    \eeq
    where
    \beq \label{Lev_set}
    \mathcal{L}(x) := \{y : 0< F^*(x)-F(x, y) \leq \gamma\, \dist(0, \partial \Psi(x))^\sigma\}.
    \eeq  
\end{enumerate}
\end{assumption}

\begin{remark}
We refer to condition \eqref{KL_y} as a local KL condition because the associated KL inequality holds only on a level set of the variable $y$, which may vary with $x$. This condition is significantly weaker than the global KL condition imposed in the literature~\cite{li2025nonsmooth, zheng2025doubly, zheng2023universal}, 
where the KL inequality is required to hold for all $y$. In contrast to the global KL condition, the local KL condition applies to a broader class of minimax problems. However, the local KL condition introduces new analytical challenges. In particular, as an optimization algorithm progresses toward a stationary point of problem~\eqref{intro_problem}, the region in which the KL condition holds may shrink, leading to a more intricate and potentially ill-conditioned landscape. As a result, addressing problem~\eqref{intro_problem} under the local KL condition requires substantially different algorithmic design and analysis.
\end{remark}

We conclude this subsection with an example of a minimax problem that satisfies the local KL condition in Assumption~\ref{ass:Lip_Smo_KL}(iii), but not the global KL condition (which requires the KL inequality to hold for all $y \in \cY$). This example demonstrates that the local KL condition applies to a broader class of minimax problems. Moreover, it illustrates that the size of the level set associated with the local KL condition depends on the outer variable $x$ and shrinks as $x$ approaches a stationary point of $\Psi$.

{\bf Example 1.} 
Consider the problem
\beq \label{ex2_pr}
\min_{1 \leq x \leq 2}\, \max_{-1 \leq y \leq 1} -(1-y^2)^2 + x(1-y^2)^3 + \frac{1}{2}(x-1)^2 - x .
\eeq
Observe that problem \eqref{ex2_pr} is a special case of \eqref{intro_problem} with $f(x, y) = -(1-y^2)^2 + x(1-y^2)^3$, $p(x) = (x-1)^2/2 - x + \delta_{[1, 2]}(x)$, and $q(y) = \delta_{[-1, 1]}(y)$. It follows from the definition of $F$ in \eqref{eq:F} that $F(x,y)= -(1-y^2)^2 + x(1-y^2)^3-\delta_{[-1, 1]}(y)$. By this and the definition of $F^*$ in \eqref{eq:F-star}, one has
\beq \label{ex2-Fstar}
F^*(x) =\max_{-1 \leq y \leq 1} -(1-y^2)^2 + x(1-y^2)^3 = \max_{0 \leq t \leq 1} -(1-t)^2 + x(1-t)^3 = \max\{x - 1, 0\}.
\eeq
Using this, \eqref{eq:Psi}, and the expression of $p$, we have $\Psi(x)=F^*(x)+p(x)=(x-1)^2/2 - 1 + \delta_{[1, 2]}(x)$, which implies that
\beq \label{ex2-parPsi}
\dist(0, \partial \Psi(x)) = x - 1 \qquad \forall x \in [1, 2].
\eeq

We first show that the local KL condition fails to hold for the inner maximization problem of \eqref{ex2_pr} on any constant level set. That is, the KL inequality \eqref{KL_y} fails to hold for any $ C > 0 $, $ \theta \in (0, 1)$, $\gamma > 0$, and $\sigma = 0$. To this end, suppose $\sigma = 0$ and fix any $\gamma > 0$. It follows from \eqref{Lev_set} that $\cL(x) = \{y: 0 < F^*(x) - F(x, y) \leq \gamma\}$. Let $\bar{x} = \min\{1+\gamma, 2\}$.  Notice that $F^*(\bar{x}) - F(\bar{x}, 1) = \min\{\gamma,1\}$ and hence $1 \in \cL(\bar{x})$. In addition, one can verify that $\dist(0, \partial_y F(\bar{x}, 1)) = 0$. Combining these, we see that the KL inequality \eqref{KL_y} fails to hold for the inner maximization problem of \eqref{ex2_pr} at $\bar{x}$ and $y = 1$ for any $C > 0$ and $\theta \in (0, 1)$. Consequently, the global KL condition also fails.

We next show that the local KL condition holds for the inner maximization problem of \eqref{ex2_pr} on a variable level set, particularly, the KL inequality \eqref{KL_y} holds with $C = 1/2$, $\theta = 1/2$, $\gamma = 1/2$, and $\sigma = 1$. To this end, fix any $x \in (1, 2]$ and $y \in \cL(x)$, where $\cL(\cdot)$ is given in \eqref{Lev_set} with $\gamma = 1/2$ and $\sigma = 1$. For convenience, let $t = 1-y^2$. It then follows from $y \in \cL(x)$,  \eqref{ex2-Fstar}, and \eqref{ex2-parPsi} that
\[
x - 1 + t^2 - x t^3= x - 1 + (1-y^2)^2 - x(1-y^2)^3 = F^*(x) - F(x, y) \leq (x-1)/2,
\]
which along with $x \in (1, 2]$ implies that $ t^2(xt - 1)=xt^3 - t^2 \geq (x-1)/2 > 0$, and hence $t > 1/x \geq 1/2$. Using $t>1/2$ and $t = 1-y^2$, we have $|y| < \sqrt{2}/2$, which along with the expression of $F$ yields $\partial_y F(x, y) = 2yt(2-3xt)$. In addition, since $t > 1/x$, one has $2 - 3xt < -1$. Using these and $t>1/2$, we obtain that $\dist(0, \partial_y F(x, y)) = 2 |y| |t| |2-3xt| \geq |y|$. On the other hand, by \eqref{ex2-Fstar},  $t = 1-y^2$, $x \leq 2$, $|y| < \sqrt{2}/2$, $t \in (1/2, 1]$, and the expression of $F$, we have
\[
F^*(x) - F(x, y) = x - 1 + t^2 - x t^3 = (1-t) \big(x(1+t+t^2) - 1 - t\big) = y^2 \big(x(1+t+t^2) - 1 - t\big) \leq y^2 (1 + t + 2t^2) \leq 4 y^2,
\]
where the inequalities follow from $x \leq 2$ and $t \in (1/2, 1]$. Hence, we can see that
\[
C (F^*(x) - F(x, y))^\theta \leq \dist(0, \partial_y F(x, y))
\]
holds with $C = 1/2$ and $\theta = 1/2$.

\section{Local generalized H\"older smoothness of the maximal function} \label{sec:Holder}

In this section, we establish a local generalized H\"older smoothness property of the maximal function $F^*$, which will play a crucial role in the development of a first-order method for solving problem \eqref{intro_problem}. 

As our goal is to develop a first-order method for computing an $\cO(\epsilon)$-stationary point of problem \eqref{intro_problem}, it is important to characterize the behavior of the objective function $\Psi$ over the following subset of nonstationary points:
\beq \label{U-eps}
\cU_\epsilon := \{x \in \cX: \operatorname{dist}(0, \partial \Psi(x)) > \epsilon\} \qquad \forall \epsilon>0.
\eeq 
Given that $p$ is a simple component of $\Psi$, it suffices to study the behavior of the more sophisticated component $F^*$ on $\cU_\epsilon$.

For a special case of problem~\eqref{intro_problem}, where $q = 0$ and the inner maximization problem of~\eqref{intro_problem} satisfies a \emph{global} PL condition (i.e., a global KL condition with exponent $1/2$), the work~\cite{nouiehed2019solving} shows that the maximal function $F^*$ is globally Lipschitz smooth. The following theorem extends this result to a more general setting, in which $q$ is a possibly nonsmooth convex function and the inner maximization problem of~\eqref{intro_problem} satisfies only a \emph{local} KL condition as described in Assumption~\ref{ass:Lip_Smo_KL}(iii). Specifically, it establishes that the maximal function $F^*$ is locally generalized H\"older smooth on $\mathcal{U}_\epsilon$. This property will play a key role in the development of an inexact proximal gradient method for solving problem~\eqref{intro_problem}. The proof of this result, which relies on an error bound for $F(x, \cdot)$, is deferred to Subsection~\ref{sec:proof1}.

\begin{theorem} \label{thm:Phi_smooth}
Let $\epsilon>0$ be given and $\cU_\epsilon$ be defined in \eqref{U-eps}. Suppose that Assumption~\ref{ass:Lip_Smo_KL} holds. Then the following statements hold.
\begin{itemize}
\item[(i)] $\partial^\rC_\cX F^*(x)$ is a singleton for all $x \in \cU_\epsilon$, and $F^*$ is differentiable on $\cU_\epsilon \cap \operatorname{int}(\cX)$. 
\item[(ii)]  For any $x, x' \in \cU_\epsilon \cap \operatorname{int}(\cX)$ satisfying $\|x - x'\| \leq \gamma \epsilon^\sigma/(2L_f)$, we have
\begin{equation} \label{thm1_HolderU}
\|\nabla F^*(x) - \nabla F^*(x')\| \leq L_{\nabla f}\|x - x'\| + (1-\theta)^{-1} {C^{-1/\theta}}L_{\nabla f}^{1/\theta}\, \|x- x'\|^{\frac{1-\theta}{\theta}}.    
\end{equation}
\item[(iii)] For any $x, x' \in \cU_\epsilon$ satisfying $\|x - x'\| \leq \gamma \epsilon^\sigma/(4L_f)$, we have
\begin{equation} \label{thm1_tilde_Holder}
\|\nabla^\rC_\cX F^*(x) - \nabla^\rC_\cX F^*(x')\| \leq L_{\nabla f}\|x - x'\| + (1-\theta)^{-1} {C^{-1/\theta}}L_{\nabla f}^{1/\theta}\, \|x- x'\|^{\frac{1-\theta}{\theta}}.    
\end{equation}
\item[(iv)] It holds that
\beq \label{thm1_nablaPhi}
\nabla^\rC_\cX F^*(x) = \nabla_x f(x, y^*) \qquad \forall x\in\cU_\epsilon, \ y^* \in Y^*(x).
\eeq
\end{itemize}

\end{theorem}

The following result is a consequence of Theorem~\ref{thm:Phi_smooth}, whose proof is deferred to Subsection~\ref{sec:proof1}.

\begin{corollary} \label{cor:Fstar-bound}
Let $\epsilon>0$ be given and $\cU_\epsilon$ be defined in \eqref{U-eps}. Suppose that Assumption~\ref{ass:Lip_Smo_KL} holds. Then, for any $x, x'$ satisfying $[x, x'] \subseteq \cU_\epsilon$ and $\|x - x'\| \leq \gamma \epsilon^\sigma/(4L_f)$, we have
\beq \label{Fstar-bound}
F^*(x) \leq F^*(x') + \langle \nabla^\rC_\cX F^*(x'), x - x' \rangle + \frac{1}{2} L_{\nabla f} \|x - x'\|^2 + \frac{M}{1+\nu} \|x - x'\|^{1+\nu},
\eeq
where 
\beq \label{Phi_smooth_constants}
M := (1-\theta)^{-1} C^{-1/\theta} L_{\nabla f}^{1/\theta}, \quad \nu := \theta^{-1}(1-\theta).
\eeq

\end{corollary}

\section{A proximal gradient method for minimizing KL function} \label{sec:subsolver}

In this section, we consider a composite optimization problem under a local KL condition:
\begin{equation} \label{pr:KL_opt}
    h^* = \min\limits_z \{ h(z) := g(z) + q(z) \},
\end{equation}
where $q: \bR^n \to \bR \cup \{\infty\}$ is closed and convex, and $g$ is $L$-smooth on $\dom\,q$. Additionally, $h$ satisfies the following local KL condition:
\beq \label{KL_subpr}
 C(h(z) - h^*)^{\theta} \leq \operatorname{dist}(0, \partial h(z)) \qquad \forall z \text{ with } h^*< h(z) \leq h^* + \delta
\eeq
for some constants $C>0$, $\theta \in [1/2, 1)$, and $\delta>0$. 

Under a global KL condition (i.e., \eqref{KL_subpr} with $\delta=\infty$), general algorithmic frameworks for solving problem~\eqref{pr:KL_opt} and analyzing their convergence properties have been extensively studied in the literature (see, e.g., \cite{attouch2013convergence,bento2025convergence,frankel2015splitting,li2018calculus}).
Inspired by these works, we propose a proximal gradient method with backtracking line search for solving problem~\eqref{pr:KL_opt} under the local KL condition \eqref{KL_subpr}, and provide a self-contained convergence analysis for completeness and ease of reference.
This algorithm will subsequently serve as a subroutine for solving problem~\eqref{intro_problem}. Specifically, at each iteration, the method performs multiple proximal gradient steps along with a backtracking line search to ensure sufficient reduction in the objective function $h$. The method terminates once the change between consecutive iterates becomes sufficiently small. The proposed method is detailed in Algorithm~\ref{algo_implement:subsolver}.

\begin{algorithm}[H]
\caption{A proximal gradient method for problem \eqref{pr:KL_opt}}\label{algo_implement:subsolver}
\begin{algorithmic}[1]
\REQUIRE $z^0 \in \{z: h(z) \leq h^* + \delta \}$, $\overline{\lambda} > 0$, $\rho \in (0,1)$, and $\tau > 0$.
\FOR{$k=0,1,2,\dots$}
    \FOR{$i=0,1,2,\dots$}
        \STATE $\lambda_{k,i} = \overline{\lambda} \rho^i$.
        \STATE $z^{k+1,i} = \mathop{\arg\min}_{z} \left\{ \langle \nabla g(z^k), z \rangle + \frac{1}{2\lambda_{k,i}} \|z - z^k\|^2 + q(z) \right\}$.
        \IF{$h(z^{k+1,i}) + \frac{1}{2\lambda_{k,i}} \|z^{k+1,i} - z^k\|^2 \leq h(z^k)$}
            \STATE $z^{k+1} = z^{k+1,i}$, $\lambda_k = \lambda_{k,i}$.
            \STATE \textbf{break}
        \ENDIF
    \ENDFOR
    \IF{$\|z^{k+1} - z^k\| \leq \tau$}
        \STATE \textbf{return} $z^{k+1}$.
    \ENDIF
\ENDFOR
\end{algorithmic}
\end{algorithm}

The following result establishes bounds on $\lambda_k$ and on the number of inner iterations performed during each outer iteration $k$. As a consequence, it justifies the well-definedness of Algorithm~\ref{algo_implement:subsolver}. The proof of this result is deferred to Subsection~\ref{sec:proof2}.  

\begin{theorem} \label{thm:sub_LS}
Let $L$ be the Lipschitz smoothness constant of $g$, $\overline{\lambda}, \rho$ be given in Algorithm~\ref{algo_implement:subsolver}, and
\[
\overline{i} = \left\lceil \frac{\log (L \overline{\lambda})}{\log \rho^{-1}} \right\rceil_+.
\]
Then it holds that the number of inner iterations of Algorithm \ref{algo_implement:subsolver} at each outer iteration $k$ is at most $\overline{i}+1$. Moreover,
\beq \label{lambda-bound}
\min\{\rho/L, \overline{\lambda}\} \leq \lambda_{k} \leq \overline{\lambda}.
\eeq
\end{theorem}

The following theorem establishes that Algorithm~\ref{algo_implement:subsolver} terminates in a finite number of iterations and yields a desired approximate solution to problem~\eqref{pr:KL_opt}. The proof is deferred to Subsection~\ref{sec:proof2}.

\begin{theorem} \label{thm:subsolver_iters}
Let $C, \delta, \theta$ be given in \eqref{KL_subpr}, $\overline{\lambda}, \rho, \tau$ be given in  Algorithm \ref{algo_implement:subsolver}, and let
\begin{align}
& \underline{\lambda} = \min\{\rho/L, \overline{\lambda}\},\qquad \underline{\beta} = \frac{C^2}{2\overline{\lambda}} \big(L + \underline{\lambda}^{-1}\big)^{-2}, \qquad \overline{\beta} = \frac{C^2}{2\underline{\lambda}} \big(L + \overline{\lambda}^{-1}\big)^{-2}, \label{thm2_const} \\
& C' = \min \left\{\frac{1}{2}, \frac{(2^{\frac{2\theta-1}{2\theta}}-1) \delta^{1-2\theta}}{(2\theta-1)\overline{\beta}}\right\}, \quad 
\qquad \overline{K}_\theta :=
\begin{cases}
\left\lceil \frac{1+\underline{\beta}}{\underline{\beta}} \log(2\overline{\lambda} \delta \tau^{-2}) \right\rceil_{+} + 1 & \text{if } \theta = \tfrac{1}{2}, \\[2ex]
\left\lceil \frac{1}{C' (2\theta-1)\underline{\beta}} \left( 2\overline{\lambda} \tau^{-2} \right)^{2\theta-1} \right\rceil + 1 & \text{if } \theta \in (\tfrac{1}{2}, 1).  \label{cor1_K}
\end{cases}  
\end{align}
Then Algorithm \ref{algo_implement:subsolver} terminates in at most $\overline{K}_\theta$ iterations, and outputs a point $z^{k+1}$ satisfying $\|z^{k+1}-z^{k}\| \leq \tau$ for some $k < \overline{K}_\theta$. Moreover, it holds that
\beq \label{h-opt}
h(z^{k+1}) - h^* \leq \big(C^{-1}(L + \underline{\lambda}^{-1})\tau\big)^{\frac{1}{\theta}}.
\eeq
\end{theorem}

\section{An inexact proximal gradient method for problem \eqref{intro_problem}} \label{sec:FOD}

In this section, we propose an inexact proximal gradient method for solving problem \eqref{intro_problem} and analyze its complexity for finding a $(\gamma \epsilon^\sigma/(4L_f),\epsilon)$-stationary point of \eqref{intro_problem} for $\epsilon >0$. 

Before proceeding, we introduce some additional notation below. Given any $\epsilon>0$, let
\beq \label{X-eps} 
\cX_\epsilon := \{x\in\cX: \dist(0, \partial\Psi(x)) \leq \epsilon\}, \quad \cX^\rc_\epsilon :=  \{x\in\cX:  \dist(x, \cX_\epsilon ) > \gamma \epsilon^\sigma / (4L_f) \},  \quad r:=\gamma \epsilon^\sigma / (4L_f),  
\eeq
where $\gamma, \sigma, L_f$ are given in Assumption \ref{ass:Lip_Smo_KL}.

To propose a method for finding an $(r,\epsilon)$-stationary point of problem \eqref{intro_problem}, we first make some key observations. Suppose $x'\in\cX^\rc_\epsilon$, that is, $x'$ is not an $(r,\epsilon)$-stationary point of \eqref{intro_problem}. Given any $x\in\cX\cap\cB(x',r)$, we observe that $[x',x] \subseteq \cX$ and moreover $\dist(0,\Psi(z))>\epsilon$ for all $z\in [x',x]$. In view of these, one can see that $[x', x] \subseteq \cU_\epsilon$, where $\cU_\epsilon$ is defined in \eqref{U-eps}. It follows from Corollary \ref{cor:Fstar-bound} that \eqref{Fstar-bound} holds for such $x$ and $x'$. In addition, notice from $\theta \in [1/2, 1)$ and \eqref{Phi_smooth_constants} that $\nu \in (0, 1]$. By this and \cite[Lemma 2]{nesterov2015universal}, one has
\[
M(1+\nu)^{-1} \|x - x'\|^{1+\nu} \leq \big( \delta^{\frac{\nu - 1}{1 + \nu}} M^{\frac{2}{1 + \nu}} \|x - x'\|^2 + \delta\big)/2 \qquad \forall \delta>0.
\]
Combining this inequality with \eqref{Fstar-bound}, and using the fact $\Psi(\cdot)=F^*(\cdot)+p(\cdot)$, we obtain that
\begin{align}
&F^*(x) \leq F^*(x') + \langle \nabla^\rC_\cX F^*(x'), x - x' \rangle + \frac{1}{2} \big(L_{\nabla f} + \delta^{\frac{\nu - 1}{1 + \nu}} M^{\frac{2}{1 + \nu}}\big) \|x - x'\|^2 + \frac{\delta}{2} \quad \forall x\in\cX\cap \cB(x',r), \label{sec4-Fstar-bound} \\
&\Psi(x) \leq F^*(x') + \langle \nabla^\rC_\cX F^*(x'), x - x' \rangle + \frac{1}{2} \big(L_{\nabla f} + \delta^{\frac{\nu - 1}{1 + \nu}} M^{\frac{2}{1 + \nu}}\big) \|x - x'\|^2 +p(x)+ \frac{\delta}{2} \quad \forall x\in\cX\cap \cB(x',r). \nn
\end{align}
As a result, when $x'\in\cX$ is not an $(r,\epsilon)$-stationary point of \eqref{intro_problem}, $\Psi$ is bounded above by a much simpler function that is the sum of a simple quadratic function and $p(\cdot)$ in a neighborhood of $x'$.

Based on the above observation, it is natural to propose a proximal gradient (PG) type-method for finding an $(r,\epsilon)$-stationary point of problem \eqref{intro_problem}. The method generates the sequence $\{x^k\}$ according to
\beq \label{prox-grad}
x^{k+1}=\mathop{\arg\min}_{x \in \cB(x^k, r)} \Big\{\langle \nabla^\rC_\cX F^*(x^k), x  \rangle + \frac{1}{2} L_k \|x - x^k\|^2 +p(x)\Big\}
\eeq
with $L_k= L_{\nabla f} + \delta_k^{(\nu-1)/(1+\nu)}M^{2/(1+\nu)}$ for a suitable choice of $\delta_k>0$, 
and terminates when $x^k$ is an $(r,\epsilon)$-stationary point of~\eqref{intro_problem} for some $k \geq 0$. However, this method faces a practical limitation: the exact value of $\nabla^\rC_\cX F^*(x^k)$ is typically unavailable, since $F^*$ is a maximal function.

To address this issue, we propose an inexact PG method for solving problem~\eqref{intro_problem}. Specifically, we replace $\nabla^\rC_\cX F^*(x^k)$ in~\eqref{prox-grad} with its approximation $\nabla_x f(x^k, y^k)$. Here, $y^k$ is an approximate solution to the subproblem $\max_y \{ f(x^k, y) - q(y) \}$, or equivalently, $\min_y \{ -f(x^k, y) + q(y) \}$, obtained via Algorithm~\ref{algo_implement:subsolver} (see lines~4 and~5 of Algorithm~\ref{algo_implement:whole}).

We now present an inexact PG method for solving problem \eqref{intro_problem}. 

\begin{algorithm}[H]
\caption{An inexact proximal gradient method for problem \eqref{intro_problem}}\label{algo_implement:whole}
\begin{algorithmic}[1]
\REQUIRE $L_f$, $L_{\nabla f}$, $C$, $\theta$, $\gamma$, $\sigma$ from Assumption~\ref{ass:Lip_Smo_KL}; $\epsilon > 0$, $\overline{\lambda} > 0$, $\rho \in (0, 1)$, and $(x^0, y^0) \in \cX \times \cY$ satisfying $F^*(x^0) - F(x^0, y^0) \le \min\{\gamma\epsilon^\sigma/2, 1\}$.

\STATE Set $r=\gamma \epsilon^\sigma / (4L_f)$, $\underline{\lambda} = \min\{\rho/L_{\nabla f}, \overline{\lambda}\}$, $M = (1-\theta)^{-1} C^{-1/\theta} L_{\nabla f}^{1/\theta}$, $\nu = \theta^{-1}(1-\theta)$.
\FOR{$k=0,1,2,\dots$}
    \STATE Set $\delta_k = 1/(k+1)$, $\eta_{k} = 1/(k+1)$,  $L_k= L_{\nabla f} + \delta_k^{(\nu-1)/(1+\nu)}M^{2/(1+\nu)}$.
    \STATE Compute 
    \beq \label{step:constrained_PG}
        x^{k+1} = \mathop{\arg\min}_{x \in \cB(x^k, r)} \Big\{ \langle \nabla_x f(x^{k}, y^{k}), x \rangle + \frac{L_k}{2} \|x - x^k\|^2 + p(x) \Big\}.
    \eeq
    \STATE Call Algorithm~\ref{algo_implement:subsolver} with $g(\cdot) \gets -f(x^{k+1}, \cdot)$, $q(\cdot) \gets q(\cdot)$, $\overline{\lambda} \gets \overline{\lambda}$, $\rho \gets \rho$, $z^0 \gets y^k$, 
    $\tau \gets \frac{C}{L_{\nabla f}+\underline{\lambda}^{-1}} \min\Big\{ (\tfrac{1}{2} \gamma \epsilon^\sigma)^{\theta}, \eta_{k+1}^{\frac{\theta}{2(1-\theta)}} \Big\}$, and denote its output as $y^{k+1}$.
\ENDFOR
\end{algorithmic}
\end{algorithm}

\begin{remark}
(i) As for the initial point $(x^0, y^0)$, Algorithm~\ref{algo_implement:whole} requires that $y^0$ be a nearly optimal solution to the problem $\max_y F(x^0, y)$. Such a pair $(x^0, y^0)$ can often be found in practice. For example, if $x^0 \in \mathcal{X}$ is such that $F(x^0, y)$ is concave in $y$, or satisfies a global or local KL condition over a known neighborhood, then $y^0$ can be obtained by applying a first-order method to the problem $\max_y F(x^0, y)$.

(ii) Some of the input parameters required by Algorithm~\ref{algo_implement:whole} may not be available in practice. It would therefore be worthwhile to develop a parameter-free or adaptive variant of Algorithm~\ref{algo_implement:whole} in future work. Alternatively, in practical implementations, one may run the algorithm with a range of trial parameter values and adjust them until the algorithm’s performance stabilizes.
\end{remark}

The following theorem establishes an iteration complexity bound for Algorithm~\ref{algo_implement:whole} to compute a $(\gamma \epsilon^\sigma/(4L_f),\epsilon)$-stationary point of problem~\eqref{intro_problem} for any $\epsilon \in (0, 1/e]$. The proof is deferred to Subsection~\ref{sec:proof3}.

\begin{theorem} \label{thm:outer_bound}
Let $L_f, L_{\nabla f}, C, \theta, \gamma, \sigma$ be given in Assumption~\ref{ass:Lip_Smo_KL}, $M, \nu$ be defined in \eqref{Phi_smooth_constants}, $\epsilon$ be given in Algorithm~\ref{algo_implement:whole}, and
\begin{align}
& A = (1-\theta)^{-2} C^{-2}L_{\nabla f}^2, \quad \underline{L} = L_{\nabla f} + M^{2/(1+\nu)}, \label{thm5_A_underL}\\
& a = 8 \big( \Psi(x^0) - \Psi^* + 3 + 2A\underline{L}^{-1} \big), \quad b = 8 \big(3/2 + A\underline{L}^{-1}\big), \label{thm5_ab} \\
& \widehat{C}_1 = \Big( 36(1+\nu)\nu^{-1}b\underline{L} \lceil \log(18 (1+\nu)\nu^{-1}b\underline{L}) \rceil_+ + 72(1+\nu)\nu^{-1}b\underline{L} + 1 \Big)^{\frac{1+\nu}{2\nu}}, \label{thm5_C1} \\
& \widehat{C}_2 = \Big( \frac{4b(1+\nu)(3M)^{2/\nu}}{M^{2/(1+\nu)}} \Big\lceil\! \log\Big(\frac{2b(1+\nu)(3M)^{2/\nu}}{M^{2/(1+\nu)}}\Big) \!\Big\rceil_+ + \frac{8b(1+\nu)(3M)^{2/\nu}}{\nu M^{2/(1+\nu)}} + 1 \Big)^{\frac{1+\nu}{2}}, \label{thm5_C2} \\
& \widehat{C}_3 = \big( 36\underline{L} a \big)^{\frac{1+\nu}{2\nu}} + M^{-1} \big( 4a(3M)^{2/\nu} \big)^{\frac{1+\nu}{2}}, \quad \widehat{C}_4 = 72A, \label{thm5_C3C4}\\
& \widehat{C}_5 = \Big( \frac{144(1+\nu)bL_{\nabla f}^2}{M^{2/(1+\nu)}} \Big\lceil\! \log\Big(\frac{72(1+\nu)bL_{\nabla f}^2}{M^{2/(1+\nu)}}\Big) \!\Big\rceil_+ + \frac{288(1+\nu)bL_{\nabla f}^2}{M^{2/(1+\nu)}} + 1 \Big)^{\frac{1+\nu}{2}}, \label{thm5_C5} \\
& \widehat{C}_6 = (144aL_{\nabla f}^2)^{\frac{1+\nu}{2}} / M, \label{thm5_C6} \\
& \widehat{C}_7 = \Big( \frac{64(1+\nu)bL_f^2}{\gamma^2 M^{2/(1+\nu)}} \Big\lceil\! \log\Big(\frac{32(1+\nu)bL_f^2}{\gamma^2 M^{2/(1+\nu)}}\Big) \!\Big\rceil_+ + \frac{128\sigma(1+\nu)bL_f^2}{\gamma^2 M^{2/(1+\nu)}} + 1 \Big)^{\frac{1+\nu}{2}}, \label{thm5_C7} \\
& \widehat{C}_8 = (64aL_f^2)^{\frac{1+\nu}{2}} / (\gamma^{1+\nu} M), \label{thm5_C8} \\
&\widehat{K}_\epsilon = \Big \lceil \widehat{C}_1 \epsilon^{-\frac{1+\nu}{\nu}} (\log \epsilon^{-1})^{\frac{1+\nu}{2\nu}} + \widehat{C}_2 \epsilon^{-\frac{1+\nu}{\nu}} (\log \epsilon^{-1})^{\frac{1+\nu}{2}} + \widehat{C}_3 \epsilon^{-\frac{1+\nu}{\nu}} + \widehat{C}_4 \epsilon^{-2} \nn \\ 
&\qquad + \widehat{C}_5 \epsilon^{-(1+\nu)} (\log \epsilon^{-1})^{\frac{1+\nu}{2}} + \widehat{C}_6 \epsilon^{-(1+\nu)} + \widehat{C}_7 \epsilon^{-(1+\nu)\sigma} (\log \epsilon^{-1})^{\frac{1+\nu}{2}} + \widehat{C}_8 \epsilon^{-(1+\nu)\sigma} \Big \rceil. \label{K-eps-hat}
\end{align}
Suppose that $\epsilon \in (0, 1/e]$. Then Algorithm~\ref{algo_implement:whole} generates a pair $(x^k, y^k)$ in at most $\widehat{K}_\epsilon$ iterations such that $x^k$ is a $(\gamma  \epsilon^\sigma/(4L_f),\epsilon)$-stationary point of problem~\eqref{intro_problem} (or equivalently the problem $\min_x \Psi(x)$), and $y^k$ satisfies
\beq \label{yk-opt}
F^*(x^k) - F(x^k, y^k) \leq \min\Big\{\frac{\gamma \epsilon^\sigma}{2}, \frac{1}{k+1} \Big\}, \quad 
\operatorname{dist}\big(y^k, Y^*(x^k)\big) \leq \frac{1}{C(1-\theta)} \min\Big\{\Big(\frac{\gamma}{2}\Big)^{(1-\theta)} \epsilon^{\sigma (1-\theta)}, \frac{1}{\sqrt{k+1}} \Big\}.
\eeq
\end{theorem}

\begin{remark} \label{rmk_iter_complexity}
In view of Theorem~\ref{thm:outer_bound}, Algorithm~\ref{algo_implement:whole} enjoys an iteration complexity of 
\[
\widetilde{\mathcal{O}}\big(\max\big\{\epsilon^{-\frac{1}{1-\theta}},\,(L_f/\gamma)^{\frac{1}{\theta}}\epsilon^{-\frac{\sigma}{\theta}}\big\} \big)
\] 
for computing a $\bigl(\gamma \epsilon^\sigma / (4L_f),\epsilon\bigr)$-stationary point of problem~\eqref{intro_problem}.
\end{remark}

The next result presents a \emph{first-order oracle complexity} bound for Algorithm~\ref{algo_implement:whole}, measured by the number of evaluations of the gradient $\nabla f$, required to generate a $(\gamma \epsilon^\sigma / (4L_f),\epsilon)$-stationary point of problem~\eqref{intro_problem} for any $\epsilon \in (0, 1/e]$. The proof is deferred to Subsection~\ref{sec:proof3}.

\begin{theorem} \label{thm:operations}
Let $\epsilon \in (0, 1/e]$ be given, $\widehat{K}_\epsilon$ be defined in Theorem~\ref{thm:outer_bound}, $L_{\nabla f}, C, \theta, \gamma, \sigma$ be given in Assumption~\ref{ass:Lip_Smo_KL}, $M, \nu$ be defined in \eqref{Phi_smooth_constants}, $\rho, \overline{\lambda}, \underline{\lambda}$ be given in Algorithm~\ref{algo_implement:whole},
and let
\begin{align}
&\underline{\beta_f} = \frac{C^2}{2\overline{\lambda}} (L_{\nabla f} + \underline{\lambda}^{-1})^{-2}, \quad
\overline{\beta_f} = \frac{C^2}{2\underline{\lambda}} (L_{\nabla f} + \overline{\lambda}^{-1})^{-2}, \nn \\
& C_f' = \min \Big\{ \frac{1}{2}, \frac{(2^{\frac{2\theta-1}{2\theta}} - 1)(\gamma\epsilon^\sigma)^{1-2\theta}}{(2\theta - 1)\overline{\beta_f}} \Big\}, \quad
\Lambda = \max\Big\{ ( \tfrac{1}{2} \gamma \epsilon^\sigma )^{-2\theta}, (\widehat{K}_\epsilon+1)^{\frac{\theta}{1-\theta}} \Big\}, \nn \\
&\overline{K}_{f, \theta} =
\begin{cases}
\left\lceil \frac{1+\underline{\beta_f}}{\underline{\beta_f}} \log(2\overline{\lambda}C^{-2}(L_{\nabla f}+\underline{\lambda}^{-1})^2 \gamma \epsilon^\sigma \Lambda) \right\rceil_+ + 1 & \text{if } \theta = \tfrac{1}{2}, \\[1ex]
\left\lceil \frac{1}{C_f'(2\theta-1)\underline{\beta_f}} \left( 2\overline{\lambda}C^{-2}(L_{\nabla f}+\underline{\lambda}^{-1})^2 \Lambda \right)^{2\theta-1} \right\rceil + 1 & \text{if } \theta \in (\tfrac{1}{2}, 1),
\end{cases} \label{cor1_Kover} \\
&\widehat{N}_\epsilon = \widehat{K}_\epsilon \Big( \Big\lceil \frac{\log (2L_{\nabla f} \overline{\lambda})}{\log \rho^{-1}} \Big\rceil_+ + 1 \Big) \overline{K}_{f, \theta}. \label{cor1_Neps}
\end{align}
Then the total number of evaluations of the proximal operators of $p$ and $q$, and the gradient $\nabla f$ performed by Algorithm~\ref{algo_implement:whole} is at most $\widehat{K}_\epsilon$, $\widehat{N}_\epsilon$, and $\widehat{K}_\epsilon + \widehat{N}_\epsilon$, respectively, to generate a pair $(x^k, y^k)$ such that $x^k$ is a $(\gamma \epsilon^\sigma/(4L_f),\epsilon)$-stationary point of problem~\eqref{intro_problem}, and $y^k$ satisfies
\eqref{yk-opt}.
\end{theorem}

\begin{remark} \label{rmk_complexity}
From Theorem~\ref{thm:operations}, we observe that, in order to compute a
$(\gamma \epsilon^\sigma /(4L_f),\epsilon)$-stationary point of
problem~\eqref{intro_problem}, Algorithm~\ref{algo_implement:whole} requires
$\widetilde{\mathcal{O}}\big(\max\big\{\epsilon^{-\frac{1}{1-\theta}},\,(L_f/\gamma)^{\frac{1}{\theta}}\epsilon^{-\frac{\sigma}{\theta}}\big\} \big)$ 
evaluations of the proximal operator of $p$. Moreover, the numbers of evaluations
of the proximal operator of $q$ and the gradient $\nabla f$ are given as follows:
{\small
\[
\left\{
\begin{array}{ll}
\widetilde{\mathcal{O}}\!\left( \max\left\{\epsilon^{-2},\,(L_f/\gamma)^2 \epsilon^{-2\sigma}\right\} \right)
& \mbox{if } \theta=\frac{1}{2}, \\[8pt]
\widetilde{\mathcal{O}}\!\left(
\begin{array}{l}
\max\Big\{(\gamma\epsilon^\sigma)^{-2\theta(2\theta-1)}\epsilon^{-\frac{1}{1-\theta}},\,
(L_f/\gamma)^{\frac{1}{\theta}}(\gamma\epsilon^\sigma)^{-2\theta(2\theta-1)}\epsilon^{-\frac{\sigma}{\theta}}, \\[4pt]
\qquad \epsilon^{-\frac{2\theta^2-2\theta+1}{(1-\theta)^2}},\,
(L_f/\gamma)^{\frac{2\theta^2-2\theta+1}{\theta(1-\theta)}}\epsilon^{-\frac{\sigma(2\theta^2-2\theta+1)}{\theta(1-\theta)}}\Big\}
\end{array}
\right)
& \mbox{if } \theta \in (\frac{1}{2},1).
\end{array}
\right.
\]
}

\end{remark}

\section{Numerical results} \label{sec:Numerical}

In this section, we conduct preliminary experiments to evaluate the performance of our proposed method (Algorithm~\ref{algo_implement:whole}).

Consider the following minimax optimization problem:
\begin{equation} \label{numerical_example}
\min_{\|x\| \leq 1} \, \max_{\|y\|_\infty \leq 2} \; \big\{0.01\|x\|_1 - \| (y + A x) \odot (y + B x) \|^2 + 0.01\,\|x - c\|^2 - 0.1\|y\|_1\big\},
\end{equation}
where $A, B \in \mathbb{R}^{m \times n}$, $c \in \mathbb{R}^n$, and $\odot$ denotes the Hadamard (elementwise) product.

For each pair $(m, n)$, we randomly generate 10 instances of problem~\eqref{numerical_example} by sampling the entries of $A$, $B$, and $c$ independently from the standard normal distribution $\mathcal{N}(0,1)$. Note that problem~\eqref{numerical_example} is a special case of problem~\eqref{intro_problem} with $f(x,y) = -\| (y + A x) \odot (y + B x) \|^2 + 0.01\|x - c\|^2$, $p(x) = 0.01\|x\|_1 + \mathcal{I}_{\cB(0,1)}(x)$, and $q(y) = 0.1\|y\|_1 + \mathcal{I}_{[-2, 2]^m}(y)$, 
where $\mathcal{I}_{\cB(0,1)}$ and $\mathcal{I}_{[-2, 2]^m}$ denote the indicator functions of the unit Euclidean ball $\cB(0,1)$ and the $m$-dimensional box $[-2, 2]^m$, respectively.

In order to apply Algorithm~\ref{algo_implement:whole} to solve problem~\eqref{numerical_example}, we need to estimate the Lipschitz constant $L_f$ of $f(\cdot,y)$ and the Lipschitz constant $L_{\nabla f}$ of $\nabla f$ over the set $\cX \times \cY$, where $\cX = \cB(0, 1)$ and $\cY = [-2, 2]^m$. To this end, let $(a^i)^T$ and $(b^i)^T$ denote the $i$th row vectors of $A$ and $B$, respectively, and define $u = y + A x$, $v = y + B x$, and $w = u \, \odot \, v$. Then we obtain that $f(x, y) = -\sum_{i=1}^m w_i^2 + 0.01 \|x - c\|^2$, and
\beq \label{numerical_gradHess}
\begin{aligned}
&\nabla_x f(x, y) = -2 \sum_{i=1}^m w_i (v_i a^i + u_i b^i) + 0.02(x - c), \quad \nabla_y f(x, y) = -2\,(u + v) \odot w, \\
&\nabla^2_{xx} f(x, y) = -2 \sum_{i=1}^m \big[ (v_i a^i + u_i b^i)(v_i a^i + u_i b^i)^T + w_i \big(a^i (b^i)^T + b^i (a^i)^T\big) \big] + 0.02 I, \\
&\nabla^2_{xy} f(x, y) = -2 \big[ A^T \diag(v^2 + 2\,u \odot v) + B^T \diag(u^2 + 2\,u \odot v) \big], \\
&\nabla^2_{yy} f(x, y) = -2\,\diag((u + v)^2 + 2\,u \odot v),
\end{aligned}
\eeq
where $z^2 := z \odot z$ for any vector $z$.  
Let $M_a = \max_i \|a^i\|$, $M_b = \max_i \|b^i\|$, and 
\beq \label{numerical_Lipconsts}
\begin{aligned}
L_f &= 4m (M_a M_b + 2M_a + 2M_b + 4)(M_a M_b + M_a + M_b) + 0.02(1 + \|c\|), \\
L_{\nabla f} &= 4m \big[ 2(M_a M_b + M_a + M_b)^2 + M_a M_b(M_a M_b + 2M_a + 2M_b + 4) \big] \\
&\quad + 2 \big[ \|A\|(M_b + 2)(2M_a + M_b + 6) + \|B\|(M_a + 2)(M_a + 2M_b + 6) \big] \\
&\quad + 2 \big[ (M_a + M_b + 4)^2 + 2(M_a + 2)(M_b + 2) \big] + 0.02,
\end{aligned}
\eeq
where $\|A\|$ and $\|B\|$ denote the spectral norms of $A$ and $B$, respectively. In view of \eqref{numerical_gradHess} and \eqref{numerical_Lipconsts}, one can verify that $L_f \geq \max\limits_{x \in \cX,\, y \in \cY} \|\nabla_x f(x, y)\|$, and 
\[
L_{\nabla f} \geq \max_{x \in \cX,\, y \in \cY} \big\{
\|\nabla^2_{xx} f(x, y)\| + \|\nabla^2_{xy} f(x, y)\| + \|\nabla^2_{yy} f(x, y)\|\big\} \geq \max_{x \in \cX,\, y \in \cY} \|\nabla^2 f(x, y)\|.
\]
It then follows that $f(\cdot, y)$ is $L_f$-Lipschitz continuous on $\cX$ for any fixed $y \in \cY$, and $\nabla f$ is $L_{\nabla f}$-Lipschitz continuous on $\cX \times \cY$.

We now apply Algorithm~\ref{algo_implement:whole} to solve problem~\eqref{numerical_example} on the randomly generated instances described above. The parameters $L_f$ and $L_{\nabla f}$ are computed using~\eqref{numerical_Lipconsts}, while the remaining parameters are set as follows: $C = 0.2$, $\theta = 0.5$, $\gamma = 0.01$, $\sigma = 0.1$, $\overline{\lambda} = 1$, $\rho = 0.95$, and $\epsilon = 10^{-2}$.\footnote{Observe that the inner maximization problem in \eqref{numerical_example} admits a decomposition into $m$ independent one-dimensional minimization subproblems, each of which has an objective function of the form $(z-a)^2(z-b)^2 + 0.1|z| + \delta_{[-2,2]}(z)$ for some $a, b \in \mathbb{R}$. Numerical evidence suggests that such functions satisfy a local KL condition with exponent $\theta = 1/2$. The other parameters were chosen empirically and yield reasonable performance in our experiments.} 
The algorithm is initialized at $(x^0, y^0) = (0, 0)$. Note that for this initialization, $y^0 = \arg\max_{\|y\|_\infty \leq 2} \{f(x^0, y) - 0.1\|y\|_1\}$, making it a suitable starting point for $y$. We run the algorithm for 10{,}000 iterations and return the final output denoted by $(x_\epsilon, y_\epsilon)$. Here, $x_\epsilon$ serves as an approximate solution to the outer minimization problem in~\eqref{numerical_example}, while $y_\epsilon$ is an approximate solution to the inner maximization problem $\max_{\|y\|_\infty \leq 2} \{f(x_\epsilon, y) - 0.1\|y\|_1\}$. 

To evaluate the performance of Algorithm~\ref{algo_implement:whole}, we compute the final objective value of problem~\eqref{numerical_example} by
\[
\Psi(x_\epsilon) = \max_{\|y\|_\infty \leq 2} \{f(x_\epsilon, y) - 0.1\|y\|_1\} + 0.01 \|x_\epsilon\|_1.
\]
Thanks to the separable structure of the problem, this maximization problem can be decomposed into $m$ independent scalar subproblems. Each subproblem is solved using the \texttt{MATLAB} subroutine \texttt{GlobalSearch}, which is a solver for finding global optima of nonconvex problems. 

In addition, we compute an estimate of the final objective value by
\[
\widehat\Psi(x_\epsilon) = f(x_\epsilon, y_\epsilon) - 0.1\|y_\epsilon\|_1 + 0.01\|x_\epsilon\|_1,
\]
using the approximate inner solution $y_\epsilon$ returned by the algorithm.

The computational results on the random instances are presented in Table~\ref{tab:results}. The first two columns list the values of $m$ and $n$. For each pair $(m,n)$, the remaining columns report the averages of the initial objective value, the final objective value $\Psi(x_\epsilon)$, and the estimated final objective value $\widehat\Psi(x_\epsilon) $ over 10 random instances. From the results, we observe that the approximate solution $x_\epsilon$ significantly reduces the objective value compared to the initial point $x^0$, and that $y_\epsilon$ is a good approximate solution to the inner maximization problem $\max_{\|y\|_\infty \leq 2} \{f(x_\epsilon, y) - 0.1\|y\|_1\}$.

\begin{table}[ht]
\centering
\caption{Numerical results for Algorithm~\ref{algo_implement:whole}}
\label{tab:results}
\begin{tabular}{c c c c c}
\toprule
$n$ & $m$ & Initial objective value & Final objective value & Estimated final objective value\\
\midrule
100 & 100 & 1.03 & -224.55 & -224.87 \\
100 & 200 & 0.98 & -228.22 & -228.69 \\
100 & 300 & 1.05 & -260.45 & -261.29 \\
200 & 100 & 1.95 & -808.08 & -808.45 \\
200 & 200 & 2.03 & -816.54 & -817.28 \\
200 & 300 & 1.95 & -837.63 & -838.33 \\
300 & 100 & 2.95 & -1102.26 & -1102.51 \\
300 & 200 & 2.90 & -1082.37 & -1082.71 \\
300 & 300 & 3.03 & -1022.22 & -1022.83 \\
\bottomrule
\end{tabular}
\end{table}

\section{Proof of the main results}\label{sec:proof}
In this section we provide a proof of our main results presented in Sections \ref{sec:Holder},  \ref{sec:subsolver},  and \ref{sec:FOD}, which are particularly Theorems \ref{thm:Phi_smooth}-\ref{thm:operations}.

\subsection{Proof of the main results in Section~\ref{sec:Holder}}\label{sec:proof1}

In this subsection we prove Theorem \ref{thm:Phi_smooth} and Corollary \ref{cor:Fstar-bound}. To proceed, we first establish several technical lemmas below.

The following lemma concerns the Lipschitz continuity of $F^*$ on $\cX$.

\begin{lemma} \label{lem:Lips}
Suppose that Assumption~\ref{ass:Lip_Smo_KL} holds. Then $F^*$ is $L_f$-Lipschitz continuous on $\cX$.
\end{lemma}

\begin{proof}
Fix any $x, x^\prime\in \cX$. Recall from Assumption~\ref{ass:Lip_Smo_KL} that $f(\cdot,y)$ is $L_f$-Lipschitz continuous on $\cX$ for any $y\in\cY$. Using this and the expression of $F$ in \eqref{eq:F}, we have 
\[
F(x, y) - F(x^\prime, y) = f(x, y) - f(x^\prime, y) \leq L_f\|x - x^\prime\| \qquad \forall y\in\cY.
\]
This together with the definition of $F^*$ in \eqref{eq:F-star} implies that
\[
F^*(x) \overset{\eqref{eq:F-star}}{=} \max_{y\in\cY} F(x, y) \leq \max_{y\in\cY} F(x^\prime, y) + L_f\|x - x^\prime\| = F^*(x^\prime) + L_f\|x- x^\prime\|,
\]
and hence $F^*(x) - F^*(x^\prime) \leq L_f\|x - x^\prime\|$. Similarly, one can show that $F^*(x^\prime) - F^*(x) \leq L_f\|x^\prime - x\|$. It then follows that $|F^*(x) - F^*(x^\prime)| \leq L_f\|x- x^\prime\|$. By this and the arbitrariness  of  $x, x^\prime\in \cX$, we conclude that $F^*$ is $L_f$-Lipschitz continuous on $\cX$.
\end{proof}

The next result provides a formula for $\nabla F^*(x)$ at a point $x$ where $F^*$ is differentiable.  
\begin{lemma} \label{lem:grad}
Suppose that Assumption~\ref{ass:Lip_Smo_KL} holds and $F^*$ is differentiable at some $x\in\bR^n$.  Then 
$\nabla F^*(x) = \nabla_x f(x, y)$ for all $y \in Y^*(x)$.
\end{lemma}

\begin{proof}
Fix any $y \in Y^*(x)$ and $d \in \mathbb{R}^n$.  Observe from \eqref{eq:F-star} that $F^*(x)=F(x, y)$ and $F^*(x + td) \geq F(x+td, y)$ for any $t\in\bR$. By these, the differentiability of $F^*$ at $x$,  and the expression of $F$, one has 
    \begin{align*}
        \langle \nabla F^*(x), d \rangle &= \lim_{t \downarrow 0} \frac{F^*(x + t d) - F^*(x)}{t} \geq \lim_{t \downarrow 0} \frac{F(x + t d, y) - F(x, y)}{t} \\
        &= \lim_{t \downarrow 0} \frac{f(x + t d, y) - f(x, y)}{t} = \langle \nabla_x f(x, y), d \rangle.
    \end{align*}
Using this and the arbitrariness  of $d$, we conclude that $\nabla F^*(x) = \nabla_x f(x, y)$.
\end{proof}

The following lemma establishes that if $F^*$ is H\"older smooth almost everywhere on a relatively open subset of $\cX$, then its differentiability extends to the entire subset.

\begin{lemma} \label{lem:diff_ext}
Let $\Gamma \subseteq \cX$ be relatively open in $\cX$, $\cS = \{ x \in \cX : F^* \text{ is differentiable at } x \}$, and $\partial^\rC_\cX F^*$ be defined as in \eqref{r-clarke}.
Suppose that Assumption~\ref{ass:Lip_Smo_KL} holds, and there exist constants $c_1, c_2 > 0$, $\alpha > 0$, and $\eta > 0$ such that
\beq \label{F-holder}
\|\nabla F^*(u) - \nabla F^*(v)\| \leq c_1 \|u - v\| + c_2 \|u - v\|^\alpha \qquad \forall\, u, v \in \cS \cap \Gamma \text{ with } \|u - v\| \leq \eta.
\eeq
Then $\partial^\rC_\cX F^*(x)$ is a singleton for every $x \in \Gamma$. Moreover, $F^*$ is differentiable on $\Gamma \cap \operatorname{int}(\cX)$.

\end{lemma}

\begin{proof}
Fix any $x \in \Gamma$. We first claim that there exists at least one sequence in $\cS$ converging to $x$. Indeed, since $\cX$ is a convex set and $\operatorname{aff}(\cX) = \bR^n$, one has $\operatorname{int}(\cX) \neq \emptyset$. Let us fix any $x' \in \operatorname{int}(\cX)$. Then, by $x \in \cX$ and the convexity of $\cX$, it follows from \cite[Proposition 1.3.1]{bertsekas2009convex} that $(x, x'] \subset \operatorname{int}(\cX)$. Consequently, there exists a sequence $\{\hx^k\} \subset \operatorname{int}(\cX)$ such that $\hx^k \to x$. This implies that for each $k$, there exists some $r_k \in (0, 1/k]$ such that $\cB(\hx^k, r_k) \subset \operatorname{int}(\cX)$. In addition, by Lemma~\ref{lem:Lips}, $F^*$ is Lipschitz continuous on the open set $\operatorname{int}(\cX)$, which implies that $\operatorname{int}(\cX) \setminus \cS$ has measure zero due to Rademacher’s theorem. In view of these, one can see that there exists some $\hat{z}^k \in \cB(\hx^k, r_k) \cap \cS$ for every $k$. Hence, we obtain
\[
\|\hat{z}^k - x\| \leq \|\hx^k - x\| + \|\hat{z}^k - \hx^k\| \leq \|\hx^k - x\| + r_k \leq \|\hx^k - x\| + 1/k \to 0    \quad \text{as } k \to \infty.
\]
It follows that $\{\hat{z}^k\} \subset \cS$ and $\hat{z}^k \to x$, which proves the claim.

Now let $\{x^k\} \subset \cS$ be an arbitrary sequence such that $x^k \to x$. Since $\Gamma$ is relatively open in $\cX$ and $x\in\Gamma$, it follows that $x^k \in \cS \cap \Gamma$ for all sufficiently large $k$. Hence, without loss of generality, we may assume that $\{x^k\} \subset \cS \cap \Gamma$.  We now claim that $\{\nabla F^*(x^k)\}$ converges. Indeed, since $\{x^k\}$ converges, it is a Cauchy sequence. Hence, there exists $K$ such that $\|x^{k} - x^{k'}\| \leq \eta$ for all $k, k' \geq  K$. By \eqref{F-holder}, one then has
\[
\|\nabla F^*(x^{k}) - \nabla F^*(x^{k'})\| \leq c_1 \|x^{k} - x^{k'}\|+ c_2 \|x^{k} - x^{k'}\|^\alpha \qquad \forall\, k, k' \geq K,
\]
which implies that $\{\nabla F^*(x^k)\}$ is also a Cauchy sequence and hence converges as claimed.  Next, we show that the limit of $\{\nabla F^*(x^k)\}$ is independent of the choice of sequence. To this end, let $\{\tx^k\} \subset \cS$ be another sequence such that $\{\tx^k\} \to x$. Interleaving $\{x^k\}$ and $\{\tx^k\}$, we obtain a sequence $\{z^k\} \subset \cS$ such that $z^k  \to x$. It then follows from the above claim that $\{\nabla F^*(z^k)\}$ converges. Since both $\{\nabla F^*(x^k)\}$ and $\{\nabla F^*(\tx^k)\}$ are subsequences of $\{\nabla F^*(z^k)\}$, they must share the same limit. Hence, the limit of $\{\nabla F^*(x^k)\}$ is independent of the sequence chosen.
In view of these and the definition of $\partial^\rC_\cX F^*(x)$ in \eqref{r-clarke}, one can see that $\partial^\rC_\cX F^*(x)$ is a singleton for any $x \in \Gamma$.

Next, let us fix any $x \in \Gamma \cap \operatorname{int}(\cX)$. By this and \eqref{r-clarke}, one can observe that
\beq \label{r-clarke-form}
\partial^\rC_\cX F^*(x) = {\rm conv}\{v : \exists\, x^k \rightarrow x  \text{ such that } \nabla F^*(x^k) \to v \}.
\eeq
Recall from Lemma \ref{lem:Lips} that $F^*$ is Lipschitz continuous on $\cX$. Moreover, note that $\Gamma \cap \operatorname{int}(\cX) \subseteq \operatorname{int}(\cX)$. It follows that the Clarke subdifferential of $F^*$ at $x$ is well defined as in \eqref{clarke}.  Using this and \eqref{r-clarke-form}, we obtain that $\partial^\rC_\cX F^*(x)=\partial^\rC F^*(x)$. Notice from above that $\partial^\rC_\cX F^*(x)$ is a singleton. Hence, $\partial^\rC F^*(x)$ is also a singleton. By this and \cite[Proposition (1.13)]{clarke1975generalized}, we conclude that $F^*$ is differentiable at any $x \in \Gamma \cap \operatorname{int}(\cX)$.
\end{proof}

The following result establishes a local $(1-\theta)^{-1}$-growth property of $F(x, \cdot)$ for any $x \in \mathcal{X}$, which was previously derived in the proof of \cite[Theorem 3.7]{drusvyatskiy2021nonsmooth}. Here, we provide an alternative and self-contained proof. Our proof generalizes the one used to derive the global quadratic growth result in \cite[Appendix G]{karimi2016linear} for the special case where $F(x, \cdot)$ satisfies the global KL condition with exponent $1/2$---that is, $F(x,\cdot)$ satisfies \eqref{KL_y} with $\theta=1/2$, and $\cL(x)$ replaced by $\cY$ for any $x\in\cX$. 

\begin{lemma} \label{lem:QG}
Suppose that Assumption~\ref{ass:Lip_Smo_KL} holds. Then it holds that for any $x \in \cX$,
\beq \label{lem4_QG}
F^*(x) - F(x, y) \geq (C(1-\theta))^{\frac{1}{1-\theta}} \operatorname{dist}(y, Y^*(x))^{\frac{1}{1-\theta}} \qquad \forall y \in \mathcal{L}(x).
\eeq
\end{lemma}
\begin{proof}
Fix any $x \in \cX$ and $y \in \mathcal{L}(x)$. It then follows from the definition of $\mathcal{L}(x)$ in \eqref{Lev_set} that $y \notin Y^*(x)$.

Recall from Assumption~\ref{ass:Lip_Smo_KL} that $f$ is Lipschitz smooth on $\cX \times \cY$. It together with the convexity of $q$ and the expression of $F$ implies that $F(x,\cdot)$ is weakly concave on $\cY$. In addition, since $y \in \mathcal{L}(x)$, one has $y \in \dom\,F(x,\cdot)$. By these and \cite[Theorem 13]{bolte2010characterizations},  there exists a unique absolutely continuous curve
$Y:[0,\infty) \to \mathbb{R}^m$ satisfying
\begin{align}
&Y(0)=y,\quad \dot{Y}(t) \in \partial_y F(x,Y(t))\quad \text{a.e. } t>0, \label{lem4_Y1} \\
& \frac{d}{dt}F(x,Y(t)) = \|\dot{Y}(t)\|^2\quad \text{a.e. } (\eta, \infty)  \label{lem4_Y2}
\end{align}
for any $\eta > 0$, and moreover, $F(x, Y(\cdot))$ is non-decreasing and continuous on $[0, \infty)$. It follows that $Y(t) \in \mathcal{L}(x)$ for any $t \geq 0$.

Let $r(t) = (F^*(x) - F(x, Y(t)))^{1-\theta}$. By $y \notin Y^*(x)$ and the monotonicity and continuity of 
$F(x, Y(\cdot))$, one can observe that $r(0) > 0$,  and $r$ is non-negative, nonincreasing, and continuous on $[0, \infty)$. We next show that
\beq \label{deri_r}
\frac{d}{dt}r(t) \le -C(1-\theta) \|\dot{Y}(t)\| \quad \text{a.e. } (\eta, \infty) 
\eeq
for any $\eta > 0$. To this end, let us fix any $\eta > 0$ and consider two separate cases below.

Case 1) $r(t) > 0$ on $[0,\infty)$. It follows from this, \eqref{KL_y}, \eqref{lem4_Y1}, and \eqref{lem4_Y2} that
\begin{align*}
\frac{d}{dt}r(t) 
&\overset{\eqref{lem4_Y2}}{=}-(1-\theta)\bigl(F^*(x)-F(x,Y(t))\bigr)^{-\theta}\|\dot{Y}(t)\|^2 \overset{\eqref{KL_y}}{\leq} -C(1-\theta)\operatorname{dist}\bigl(0,\partial_y F(x,Y(t))\bigr)^{-1}\|\dot{Y}(t)\|^2\\ 
&\overset{\eqref{lem4_Y1}}{\leq} -C(1-\theta)\,\|\dot{Y}(t)\|^{-1}\|\dot{Y}(t)\|^2 = -C(1-\theta)\,\|\dot{Y}(t)\| \quad \text{a.e. } (\eta, \infty),
\end{align*}
and hence \eqref{deri_r} holds as desired.

Case 2) $r(t) = 0$ for some $t > 0$.  Since $r$ is continuous on $[0, \infty)$,  one has that $t_0:=\min\{t>0: r(t)=0\}>0$.  By this and the nonnegativity and monotonicity of $r$, we have $r(t)=0$ and hence $F(x, Y(t))=F^*(x)$ for all $t \ge t_0$.  It then follows from \eqref{lem4_Y2} with $\eta$ replaced by $t_0$ that $\|\dot{Y}(t)\|= 0$ almost everywhere on $(t_0, \infty)$. Hence, we obtain
\beq \label{deri_r1}
\frac{d}{dt}r(t) \le -C(1-\theta) \|\dot{Y}(t)\| \quad \text{a.e. } (t_0, \infty).
\eeq
It follows that \eqref{deri_r} holds if $\eta \geq t_0$. We now assume that $\eta <t_0$. Note that $r(t)>0$ for all $t\in [\eta,t_0)$. By a similar argument as in Case 1), one can conclude that 
\[
\frac{d}{dt}r(t) \le -C(1-\theta) \|\dot{Y}(t)\| \quad \text{a.e. } (\eta, t_0).
\]
Combining this with \eqref{deri_r1}, we see that \eqref{deri_r} holds in this case as well.

Fix any $T > 0$ and $\delta \in (0, T)$. By \eqref{deri_r}, the monotonicity of $r$, and the absolute continuity of $Y(\cdot)$, one has
\begin{align*}
    r(T) - r(\delta) &\leq \int_{\delta}^{T} \frac{d}{dt} r(t) \, dt \overset{\eqref{deri_r}}{\leq} -C(1-\theta) \int_{\delta}^{T} \|\dot{Y}(t)\| \, dt 
    \leq -C(1-\theta) \left\| \int_{\delta}^{T} \dot{Y}(t) \, dt \right\| \\ 
    &= -C(1-\theta) \|Y(T) - Y(\delta)\|,
\end{align*}
where the first inequality follows from the monotonicity of $r$ (e.g., see \cite[Chapter~3, Exercise~16]{stein2009real}), and the equality uses the absolute continuity of $Y(\cdot)$. Taking the limit on both sides of the above relation as $\delta \to 0$, and using $Y(0)=y$ and the continuity of $r$ and $Y(\cdot)$, we obtain  
\begin{equation} \label{eq:QG1}
    r(T) - r(0) \leq -C(1-\theta) \|Y(T) - y\|.
\end{equation}

We next show that  $\lim_{T \to \infty} r(T) = 0$. It clearly holds if there exists some $t > 0$ such that $r(t) = 0$, due to the nonnegativity and monotonicity of $r$. We now assume that $r(t) > 0$ for all $t\in [0,\infty)$. By this, \eqref{KL_y}, \eqref{lem4_Y1}, \eqref{lem4_Y2}, and the monotonicity of $r$, one has that  for any $T > 0$ and $\delta \in (0, T)$,
\begin{align*}
    r(T) - r(\delta) &\leq \int_{\delta}^{T} \frac{d}{dt} r(t) \, dt \overset{\eqref{lem4_Y2}}{=} -(1-\theta) \int_{\delta}^{T} (F^*(x) - F(x, Y(t)))^{-\theta} \|\dot{Y}(t)\|^2 \, dt \\
    &\overset{\eqref{lem4_Y1}}{\leq} -(1-\theta) \int_{\delta}^{T} (F^*(x) - F(x, Y(t)))^{-\theta} \operatorname{dist}(0, \partial_y F(x, Y(t)))^2 \, dt \\
    &\overset{\eqref{KL_y}}{\leq} -C^2(1-\theta) \int_{\delta}^{T} (F^*(x) - F(x, Y(t)))^{\theta} \, dt  = -C^2(1-\theta) \int_{\delta}^{T} r(t)^{\frac{\theta}{1-\theta}} \, dt \\ 
    & \leq -C^2(1-\theta) (T - \delta) r(T)^{\frac{\theta}{1-\theta}}, 
\end{align*}
where the first and last inequalities follow from the monotonicity of $r$. This relation and  $r(T)>0$ imply that 
$ - r(\delta) \leq -C^2(1-\theta) (T - \delta) r(T)^{\frac{\theta}{1-\theta}}$. 
Taking the limit on both sides of this relation as $\delta \to 0$, and using the continuity of $r$, we obtain that 
$- r(0) \leq -C^2(1-\theta) T r(T)^{\frac{\theta}{1-\theta}}$, which yields $r(T) \leq [r(0)/(C^2(1-\theta) T)]^{\frac{1-\theta}{\theta}}$. This along with $r(T)>0$ implies that $r(T) \to 0$ as $T \to \infty$.

By \eqref{eq:QG1} and the nonnegativity of $r$, one can observe that the range of $Y(\cdot)$ is bounded. In addition, notice from Assumption \ref{ass:Lip_Smo_KL} that $\dom\,F(x, \cdot)$ is closed.  By these facts, there exists a sequence $\{t_k\} \subset (0,\infty)$ such that $t_k\to\infty$ and $\{Y(t_k)\}$ converges to some point $y^* \in \operatorname{dom} F(x, \cdot)$.  Recall that $\lim_{t \to \infty} r(t) = 0$, which along with $t_k\to\infty$ implies that $r(t_k) \to 0$.  It then follows that $\lim_{k \to \infty} F(x, Y(t_k)) = F^*(x)$. On the other hand,  by the upper semicontinuity of $F(x, \cdot)$ and $Y(t_k) \to y^* \in \operatorname{dom} F(x, \cdot)$, one has $\limsup_{k \to \infty} F(x, Y(t_k)) \leq F(x, y^*)$. Combining these relations, we conclude that $y^* \in Y^*(x)$. Finally, letting $T = t_k$ in \eqref{eq:QG1}, one has $r(t_k) - r(0) \leq -C(1-\theta) \|Y(t_k) - y\|$. Taking the limit on both sides of this inequality as $k\to\infty$, and using the fact that  $r(t_k) \to 0$ and $Y(t_k)\to y^* \in Y^*(x)$, we obtain that
\[
r(0) \geq  C(1-\theta) \|y^* - y\| \geq C(1-\theta) \operatorname{dist}(y, Y^*(x)),
\]
which together with the expression of $r$ implies that the conclusion \eqref{lem4_QG} holds.
\end{proof}

As an immediate consequence of Lemma \ref{lem:QG} and Assumption \ref{ass:Lip_Smo_KL}, the following lemma establishes an error bound for $F(x,\cdot)$. This result, originally derived in \cite[Theorem 3.7]{drusvyatskiy2021nonsmooth}, provides a relationship between $\dist(y, Y^*(x))$ and $\dist(0, \partial_y F(x, y))$.

\begin{lemma} \label{lem:EB}
Suppose that Assumption~\ref{ass:Lip_Smo_KL} holds. Then it holds that for any $x \in \cX$, 
\beq \label{lem5_EB}
\operatorname{dist}(y, Y^*(x)) \leq (1-\theta)^{-1} {C^{-\frac{1}{\theta}}} \operatorname{dist}\big(0, \partial_y F(x, y)\big)^{\frac{1-\theta}{\theta}} \qquad \forall y \in \mathcal{L}(x).
\eeq
\end{lemma}

\begin{proof}
The relation \eqref{lem5_EB} follows from \eqref{KL_y} and \eqref{lem4_QG}.
\end{proof}

The next result concerns the relative openness of the set $\cU_\epsilon$ in $\cX$.

\begin{lemma} \label{lem:open}
Let $\cU_\epsilon$ be defined in \eqref{U-eps}. Suppose that Assumption~\ref{ass:Lip_Smo_KL} holds. Then $\cU_\epsilon$ is relatively open in $\cX$ for any $\epsilon > 0$.
\end{lemma}

\begin{proof}
Fix any $\epsilon > 0$. To prove that $\cU_\epsilon$ is relatively open in $\cX$, it suffices to show that the set $\{x \in \mathbb{R}^n : \operatorname{dist}(0, \partial \Psi(x)) > \epsilon\}$ is open, or equivalently,  $\cC=\{x \in \mathbb{R}^n : \operatorname{dist}(0, \partial \Psi(x)) \leq \epsilon\}$ is closed.
To this end, consider any convergent sequence $\{x^k\} \subset \cC$ with $x^k \to x$ for some $x\in\bR^n$.  Clearly,  $\cC \subseteq \cX$ and hence $\{x^k\} \subset \cX$. It then follows from $x^k \to x$ and the closedness of $\cX$ that $x\in\cX$.  Also, since $x^k \in \cC$, there exists $s^k \in \partial \Psi(x^k)$ with $\|s^k\| \leq \epsilon$. Without loss of generality, we may assume that $s^k \to s$ for some $s$ with $\|s\| \leq \epsilon$.
Observe from Assumption~\ref{ass:Lip_Smo_KL} that $f(\cdot, y)$ is $L_{\nabla f}$-smooth on $\cX$ for all $y \in \cY$, which together with \eqref{eq:F-star} implies that $F^*$ is weakly convex on $\cX$. By this, the convexity of $p$, $\operatorname{ri}(\dom\,F^*) \cap \operatorname{ri}(\dom\,p)=\operatorname{int}(\cX) \neq \emptyset$, $\Psi=F^*+p$, and \cite[Theorem 28.3]{rockafellar1997convex}, one has $\partial \Psi=\partial F^*+\partial p$ on $\cX$.
Additionally, notice that $\partial F^*$ and $\partial p$ are outer semicontinuous on $\cX$. Hence, $\partial \Psi$ is outer semicontinuous on $\cX$, which together with $x^k \to x$, $s^k \in \partial \Psi(x^k)$, and $s^k \to s$ implies that $s \in \partial \Psi(x)$. By this and $\|s\| \leq \epsilon$, we conclude that $x\in\cC$. Hence, $\cC$ is closed as desired.
\end{proof}

We are now ready to prove Theorem~\ref{thm:Phi_smooth}.

\begin{proof}[\textbf{Proof of Theorem~\ref{thm:Phi_smooth}}] 
(i) Fix any $\epsilon > 0$ and $x \in \cU_\epsilon$. It follows that $x\in\cX$ and $\operatorname{dist}(0, \partial \Psi(x)) > \epsilon$.  This together with \eqref{Lev_set} implies that $\{y: F^*(x) > F(x, y) \geq F^*(x) - \gamma \epsilon^\sigma\} \subseteq \cL(x)$. By this and \eqref{lem5_EB}, one has
\beq \label{thm1_EB}
\operatorname{dist}(y, Y^*(x)) \leq (1-\theta)^{-1} {C^{-\frac{1}{\theta}}} \operatorname{dist}(0, \partial_y F(x, y))^{\frac{1-\theta}{\theta}} \qquad \forall y\,  \text{ with } F^*(x) > F(x, y) \geq F^*(x) - \gamma \epsilon^\sigma.
\eeq
Clearly, the above relation also holds for any $y\in Y^*(x)$. Now, fix any $x' \in \cU_\epsilon$ with $\|x - x'\| \leq \gamma \epsilon^\sigma/(2L_f)$. Observe from  Assumption~\ref{ass:Lip_Smo_KL} that $Y^*(x') \neq \emptyset$. Let $y^*(x') \in Y^*(x')$ be arbitrarily chosen. Then $F(x', y^*(x')) = F^*(x')$. By these, Assumption~\ref{ass:Lip_Smo_KL}, and Lemma~\ref{lem:Lips}, one has
\begin{align*}
    F(x, y^*(x')) - F^*(x) = F(x, y^*(x')) - F(x', y^*(x'))  + F^*(x') - F^*(x) 
    \geq -2L_f\|x - x'\| \geq -\gamma \epsilon^\sigma,
\end{align*}
where the first inequality uses the $L_f$-Lipschitz continuity of $F^*$  and $F(\cdot, y)$ for each $y \in \mathcal{Y}$ due to Assumption~\ref{ass:Lip_Smo_KL} and Lemma~\ref{lem:Lips}. Hence, it follows from \eqref{thm1_EB} that
\[
\operatorname{dist}(y^*(x'), Y^*(x)) \leq (1-\theta)^{-1} {C^{-\frac{1}{\theta}}} \operatorname{dist}(0, \partial_y F(x, y^*(x')))^{\frac{1-\theta}{\theta}}.    
\]

Since $y^*(x') \in Y^*(x')$, by the first-order optimality condition, one has $0 \in \partial_y F(x', y^*(x'))$. In addition, by the expression of $F$ and the smoothness of $f$ on $\cX \times \cY$, we obtain
\beq \label{partial-y}
\partial_y F(x', y^*(x')) = \nabla_y f(x', y^*(x')) - \partial q(y^*(x')),\qquad 
\partial_y F(x, y^*(x')) = \nabla_y f(x, y^*(x')) - \partial q(y^*(x')).
\eeq
The first relation in \eqref{partial-y} and $0 \in \partial_y F(x', y^*(x'))$ lead to $ \nabla_y f(x', y^*(x')) \in 
\partial q(y^*(x'))$, which along with the second relation in \eqref{partial-y} implies that
\[
\nabla_y f(x, y^*(x')) - \nabla_y f(x', y^*(x')) \in \partial_y F(x, y^*(x')).
\]
Using this and the Lipschitz continuity of $\nabla_y f$, we have
\[
\operatorname{dist}(0, \partial_y F(x, y^*(x'))) \leq \|\nabla_y f(x, y^*(x')) - \nabla_y f(x', y^*(x'))\| \leq L_{\nabla f} \|x' - x\|.
\]
Combining this with \eqref{thm1_EB} yields
\[
\operatorname{dist}(y^*(x'), Y^*(x)) \leq (1-\theta)^{-1} {C^{-\frac{1}{\theta}}} L_{\nabla f}^{\frac{1-\theta}{\theta}} \|x' - x\|^{\frac{1-\theta}{\theta}}.
\]
Notice from Assumption~\ref{ass:Lip_Smo_KL} that $Y^*(x)$ is a nonempty closed set. Hence, there exists $y^*(x) \in Y^*(x)$ such that $\|y^*(x')-y^*(x)\| = \operatorname{dist}(y^*(x'), Y^*(x))$. By this and the above relation, one has
\beq \label{thm1_distWithx1x2}
\|y^*(x')-y^*(x)\|   \leq (1-\theta)^{-1} {C^{-\frac{1}{\theta}}} L_{\nabla f}^{\frac{1-\theta}{\theta}} \|x' - x\|^{\frac{1-\theta}{\theta}}.
\eeq   

Suppose further that $F^*$ is differentiable at $x$ and $x'$. Using this, $y^*(x) \in Y^*(x)$, $y^*(x') \in Y^*(x')$, and Lemma~\ref{lem:grad}, we obtain 
\[
\nabla F^*(x) = \nabla_x f(x, y^*(x)), \qquad \nabla F^*(x') = \nabla_x f(x', y^*(x')). 
\] 
By these, \eqref{thm1_distWithx1x2}, $\theta \in [1/2,1)$, $\|x - x'\| \leq \gamma \epsilon^\sigma/(2L_f)$, and the Lipschitz smoothness of $f$, one has
\begin{align}
    &\| \nabla F^*(x') - \nabla F^*(x) \| = \| \nabla_x f(x', y^*(x')) - \nabla_x f(x, y^*(x)) \| \nn \\ 
    &\leq \| \nabla_x f(x', y^*(x')) - \nabla_x f(x, y^*(x'))\| + \|\nabla_x f(x, y^*(x')) - \nabla_x f(x, y^*(x)) \| \nn\\
    &\leq L_{\nabla f}\|x' - x\| + L_{\nabla f} \|y^*(x') - y^*(x)\| \overset{\eqref{thm1_distWithx1x2}}{\leq} L_{\nabla f}\|x' - x\| + (1-\theta)^{-1} {C^{-\frac{1}{\theta}}}L_{\nabla f}^{\frac{1}{\theta}} \|x' - x\|^{\frac{1-\theta}{\theta}}. \label{thm1_holder_bound1} 
\end{align}
Using this, Lemma~\ref{lem:diff_ext}, and the fact that $\cU_\epsilon$ is relatively open in $\cX$ (see Lemma~\ref{lem:open}), we conclude that $\partial^\rC_\cX F^*(x)$ is a singleton for every $x \in \cU_\epsilon$ and $F^*$ is differentiable on $\cU_\epsilon \cap \operatorname{int}(\cX)$. Hence, statement (i) of Theorem~\ref{thm:Phi_smooth} holds.

(ii) The relation \eqref{thm1_HolderU} directly follows from statement (i) of Theorem~\ref{thm:Phi_smooth} and \eqref{thm1_holder_bound1}. This proves statement (ii) of Theorem~\ref{thm:Phi_smooth}.

(iii) We now prove \eqref{thm1_tilde_Holder}. Let us fix any $x, x' \in \cU_\epsilon$ with $\|x - x'\| \leq \gamma \epsilon^\sigma / (4L_f)$. It follows from statement (i) of Theorem~\ref{thm:Phi_smooth} that both $\partial^\rC_\cX F^*(x)$ and $\partial^\rC_\cX F^*(x')$ are a singleton. By this and \eqref{r-clarke}, there exist sequences $\{x^k\} \subset \cS$ and $\{\hx^k\} \subset \cS$ such that $\nabla^\rC_\cX F^*(x) = \lim_{k\to\infty} \nabla F^*(x^k)$, $\nabla^\rC_\cX F^*(x') = \lim_{k\to\infty} \nabla F^*(\hx^k)$, $x^k \to x$, and $\hx^k \to x'$. In view of these and the relative openness of $\cU_\epsilon$, we may assume without loss of generality that $x^k, \hx^k \in \cU_\epsilon \cap \cS$ and $\|x^k - \hx^k\| \leq \gamma \epsilon^\sigma / (2L_f)$ for all $k$. It then follows from \eqref{thm1_holder_bound1} that
\[
\| \nabla F^*(x^k) - \nabla F^*(\hx^k) \| \leq L_{\nabla f}\|x^k - \hx^k\| + (1-\theta)^{-1} {C^{-\frac{1}{\theta}}}L_{\nabla f}^{\frac{1}{\theta}} \|x^k - \hx^k\|^{\frac{1-\theta}{\theta}}.
\]
Taking limits on both sides of this inequality as $k \to \infty$ yields \eqref{thm1_tilde_Holder}. Hence, statement (iii) of Theorem~\ref{thm:Phi_smooth} holds. 

(iv) We next prove \eqref{thm1_nablaPhi}. To this end, fix any $x \in \cU_\epsilon$ and $y^* \in Y^*(x)$. It follows from statement (i) of Theorem~\ref{thm:Phi_smooth} that $\partial^\rC_\cX F^*(x)$ is a singleton. By this and \eqref{r-clarke},  there exists a sequence $\{x^k\} \subset \cS$ such that $\nabla^\rC_\cX F^*(x) = \lim_{k\to\infty} \nabla F^*(x^k)$ and $x^k \to x$. Using this and the relative openness of $\cU_\epsilon$, we may assume without loss of generality that $x^k \in \cU_\epsilon$ and $\|x^k - x\| \leq \gamma \epsilon^\sigma/(2L_f)$ for all $k$. By these and  \eqref{thm1_distWithx1x2},  there exists some $y^k \in Y^*(x^k)$ such that 
\[
\|y^k - y^*\| \leq (1-\theta)^{-1} {C^{-\frac{1}{\theta}}} L_{\nabla f}^{\frac{1-\theta}{\theta}} \|x^k - x\|^{\frac{1-\theta}{\theta}}.
\]
This along with $x^k \to x$ implies $y^k \to y^*$. Also, notice from $x^k \in \cS$, $y^k \in Y^*(x^k)$, and Lemma~\ref{lem:grad} that $\nabla F^*(x^k) = \nabla_x f(x^k, y^k)$ for all $k$. In view of these and the  continuity of $\nabla_x f$, one has $\nabla^\rC_\cX F^*(x) = \nabla_x f(x, y^*)$, and hence \eqref{thm1_nablaPhi} holds. This proves statement (iv) of Theorem~\ref{thm:Phi_smooth}.
\end{proof}

We finally prove Corollary~\ref{cor:Fstar-bound}.

\begin{proof}[\textbf{Proof of Corollary~\ref{cor:Fstar-bound}}] 
Fix any $x, x'$ satisfying $[x, x'] \subseteq \cU_\epsilon$ and $\|x - x'\| \leq \gamma \epsilon^\sigma / (4L_f)$. By $\operatorname{int}(\cX) \neq \emptyset$, $x, x' \in \cX$, the convexity of $\cX$, and a similar argument as used in the proof of Lemma~\ref{lem:diff_ext}, one can see that there exist sequences $\{x^k\} \subset \operatorname{int}(\cX)$ and $\{\hx^k\} \subset \operatorname{int}(\cX)$ such that $x^k \to x$ and $\hx^k \to x'$. Then, one can deduce from $[x, x'] \subseteq \cU_\epsilon$ and the relative openness of $\cU_\epsilon$ that $[x^k, \hx^k] \subseteq \cU_\epsilon$ for all sufficiently large $k$. In addition, notice from the convexity of $\operatorname{int}(\cX)$ that $[x^k, \hx^k] \subseteq \operatorname{int}(\cX)$ for all $k$. In view of these, we may assume without loss of generality that $[x^k, \hx^k] \subseteq \cU_\epsilon \cap \operatorname{int}(\cX)$ and $\|x^k - \hx^k\| \leq \gamma \epsilon^\sigma / (2L_f)$ for all $k$. It then follows from Theorem~\ref{thm:Phi_smooth}(ii) that for any $z, z' \in [x^k, \hx^k]$,
\[
\|\nabla F^*(z) - \nabla F^*(z')\| \leq L_{\nabla f}\|z - z'\| + (1-\theta)^{-1} {C^{-1/\theta}}L_{\nabla f}^{1/\theta}\, \|z - z'\|^{\frac{1-\theta}{\theta}}.   
\]
Hence, we obtain that
\beq \label{cor1_Fstar_xk}
F^*(x^k) \leq F^*(\hx^k) + \langle \nabla F^*(\hx^k), x^k - \hx^k \rangle + \frac{1}{2} L_{\nabla f} \|x^k - \hx^k\|^2 + \frac{M}{1+\nu} \|x^k - \hx^k\|^{1+\nu},
\eeq
where $M$ and $\nu$ are defined in \eqref{Phi_smooth_constants}. Note that $\{\hx^k\} \subset \operatorname{int}(\cX)$ and $F^*$ is Lipschitz continuous on $\cX$ (see Lemma \ref{lem:Lips}). This implies that $\{\nabla F^*(\hx^k)\}$ is bounded.  By this, $\{\hx^k\} \subset \cX$, $\hx^k \to x'$, and \eqref{r-clarke}, one can observe that any limit point of $\{\nabla F^*(\hx^k)\}$ belongs to $ \partial^\rC_\cX F^*(x')$. Moreover,  $\partial^\rC_\cX F^*(x')$ is a singleton due to Theorem~\ref{thm:Phi_smooth} and $x'\in \cU_\epsilon$ . Consequently, the limit of $\{\nabla F^*(\hx^k)\}$ as $k\to \infty$ exists, and moreover, $ \lim_{k\to\infty} \nabla F^*(\hx^k)=\nabla^\rC_\cX F^*(x')$. Using this, $x^k \to x$, $\hx^k \to x'$, the continuity of $F^*$, and taking limits on both sides of \eqref{cor1_Fstar_xk} as $k \to \infty$, we conclude that \eqref{Fstar-bound} holds, which proves Corollary~\ref{cor:Fstar-bound}.
\end{proof}

\subsection{Proof of the main results in Section~\ref{sec:subsolver}}\label{sec:proof2}

In this subsection, we provide the proofs of Theorems~\ref{thm:sub_LS} and~\ref{thm:subsolver_iters}. We begin by proving Theorem~\ref{thm:sub_LS}.

\begin{proof}[\textbf{Proof of Theorem~\ref{thm:sub_LS}}] 
Suppose for contradiction that the inner loop runs for more than $\overline{i}+1$ iterations at the $k$th outer iteration. Then one can observe from Algorithm \ref{algo_implement:subsolver} that
\beq \label{aux-ineq1}
h(z^{k+1,\overline{i}}) + \frac{1}{2\lambda_{k,\overline{i}}} \|z^{k+1,\overline{i}} - z^k\|^2 > h(z^k).
\eeq
By the optimality condition for $z^{k+1, \overline{i}}$, one has 
\[
\langle \nabla g(z^k), z^{k+1, \overline{i}}-z^k \rangle + \frac{1}{\lambda_{k,i}} \|z^{k+1, \overline{i}} - z^k\|^2 + q(z^{k+1, \overline{i}}) \leq q(z^k).
\]
In addition, by the $L$-smoothness of $g$, we have
\[
g(z^{k+1, \overline{i}}) \leq g(z^k) + \langle \nabla g(z^k), z^{k+1, \overline{i}} - z^k \rangle + \frac{L}{2} \|z^{k+1, \overline{i}} - z^k\|^2.
\]
Combining these two inequalities yields
\beq \label{aux-ineq2}
h(z^{k+1, \overline{i}}) + \Big( \frac{1}{\lambda_{k, \overline{i}}} - \frac{L}{2} \Big) \|z^{k+1, \overline{i}} - z^k\|^2 \leq h(z^k).
\eeq
By the definition of $\overline{i}$ and $\lambda_{k,\overline{i}} = \overline{\lambda} \rho^{\overline{i}}$, one has $L \leq 1/\lambda_{k, \overline{i}}$. This and \eqref{aux-ineq2} imply that
\[
h(z^{k+1, \overline{i}}) + \frac{1}{2\lambda_{k, \overline{i}}} \|z^{k+1, \overline{i}} - z^k\|^2 \leq h(z^k),
\]
which contradicts \eqref{aux-ineq1}. Hence, the inner loop runs at most $\overline{i}$+1 iterations. By this and the definition of $\lambda_k$, one has  
$\min\{\rho/L, \overline{\lambda}\} \leq \overline{\lambda} \rho^{\overline{i}} \leq\lambda_k \leq \overline{\lambda}$. Hence, the conclusion of Theorem~\ref{thm:sub_LS} holds.
\end{proof}

In the remainder of this subsection, we present the proof of Theorem~\ref{thm:subsolver_iters}. To this end, we first establish several technical lemmas. The following result provides a bound on $\operatorname{dist}(0, \partial h(z^{k+1}))$ in terms of $\|z^{k+1} - z^k\|$.

\begin{lemma} \label{lem:H2_bound}
Let $z^k$ and $z^{k+1}$ be generated by Algorithm \ref{algo_implement:subsolver} for some $k\geq 0$. Then it holds that
\beq \label{lem8}
\operatorname{dist}(0, \partial h(z^{k+1})) \leq \big(L + \lambda_k^{-1}\big)
\|z^{k+1} - z^k\|.
\eeq
\end{lemma}

\begin{proof}
By the optimality condition for $z^{k+1}$, one has
\[
0 \in \nabla g(z^k) + \lambda_k^{-1}(z^{k+1}-z^k) + \partial q(z^{k+1}),
\]
which implies that
\[
\nabla g(z^{k+1}) - \nabla g(z^k) - \lambda_k^{-1}(z^{k+1}-z^k) \in \partial h(z^{k+1}).
\]
Using this and the $L$-smoothness of $g$, we obtain
\[
\operatorname{dist}(0, \partial h(z^{k+1})) \leq \|\nabla g(z^{k+1}) - \nabla g(z^k) - \lambda_k^{-1}(z^{k+1}-z^k)\| 
\leq (L + \lambda_k^{-1})\|z^{k+1}-z^k\|,
\]
and hence the conclusion holds.
\end{proof}

For notational convenience, let
\beq \label{rk}
a_k = \frac{1}{2\lambda_k}, \qquad 
b_k = \left(L + \frac{1}{\lambda_k}\right)^{-1}.
\eeq
In view of these, Algorithm \ref{algo_implement:subsolver}, Theorem~\ref{thm:sub_LS}, and Lemma~\ref{lem:H2_bound}, one can observe that the following relations hold:
\begin{align} 
&h(z^{k+1}) + a_k \|z^{k+1} - z^k\|^2 \leq h(z^k), \label{H1} \\
&b_k \operatorname{dist}(0, \partial h(z^{k+1})) \leq \|z^{k+1} - z^k\|, \label{H2}  \\ 
& (2\overline{\lambda})^{-1} \leq a_k \leq (2\underline{\lambda})^{-1}, \qquad (L + \underline{\lambda}^{-1})^{-1} \leq b_k \leq (L + \overline{\lambda}^{-1})^{-1},\label{H3}
\end{align}
where $\underline{\lambda}$ is defined in \eqref{thm2_const}. 
In addition, by \eqref{H1} and the choice of $z^0$, we can observe that $r_0 \leq \delta$, and $\{r_k\}$ is nonincreasing. Consequently, $r_k \leq \delta$ holds for all $k$.

The following lemma establishes a convergence rate for Algorithm \ref{algo_implement:subsolver}, following a similar argument as in~\cite[Theorem 3.4]{frankel2015splitting}.

\begin{lemma} \label{lem:subsolver_converge}
Let $\delta, \theta, \underline{\lambda}, \underline{\beta}, \overline{\beta}, C'$  and $\overline{\lambda}$ be given in \eqref{KL_subpr}, \eqref{thm2_const}, \eqref{cor1_K}, \eqref{rk} and Algorithm \ref{algo_implement:subsolver}, respectively. Suppose that $z^k$ is generated by Algorithm \ref{algo_implement:subsolver} for some $k\geq 1$. Then the following statements hold.
\begin{enumerate} [label=(\roman*)]
\item If $\theta = 1/2$, then
\beq \label{thm2_linear}
   h(z^k) - h^*  \leq \delta e^{-\frac{\underline{\beta}}{1+\underline{\beta}} k}.
\eeq
\item If $\theta \in (1/2, 1)$, then
\beq \label{thm2_subli}
h(z^k) - h^* \leq \left( \frac{1}{C' (2\theta-1)\underline{\beta}} \right)^{\frac{1}{2\theta-1}} k^{-\frac{1}{2\theta-1}}.
\eeq
\end{enumerate}
\end{lemma}

\begin{proof}
For notational convenience, let $r_\ell = h(z^\ell) - h^*$ for all $\ell$. Since $h(z^0)-h^* \leq \delta$ and $\{h(z^\ell)\}$ is nonincreasing, \eqref{KL_subpr} holds with $z=z^\ell$ for all $\ell \geq 0$. 
By this, \eqref{H1}, and \eqref{H2}, one has  
\[
r_\ell - r_{\ell+1} \overset{\eqref{H1}}{\geq} a_\ell \|z^{\ell+1} - z^\ell\|^2 \overset{\eqref{H2}}{\geq} a_\ell b_\ell^2 \operatorname{dist}\big(0, \partial h(z^{\ell+1})\big)^2 \overset{\eqref{KL_subpr}}{\geq} {a_\ell b_\ell^2} C^2 r_{\ell+1}^{2\theta}.
\]
Let $\beta_\ell := {a_\ell b_\ell^2} C^2$ for all $\ell$. Using \eqref{thm2_const} and \eqref{H3}, we have
\beq \label{thm2_recur}
r_\ell - r_{\ell+1} \geq \beta_\ell r_{\ell+1}^{2\theta}, \qquad \beta_\ell \in [\underline{\beta}, \overline{\beta}],
\eeq

(i) Suppose $\theta = 1/2$. It then follows from \eqref{thm2_recur} that $r_{\ell+1} \leq (1+\beta_\ell)^{-1} r_\ell$ for all $\ell$. Hence,  
\beq \label{rk-bound}
r_k \leq r_0 \prod_{\ell=0}^{k-1} (1+\beta_\ell)^{-1} \qquad \forall k \geq 0.
\eeq
By the concavity of $\log(\cdot)$,  one has that $ \log(1+t) \leq t $ for all $ t > -1$.  It follows that 
\[
\log (1+\beta_\ell)^{-1} = \log\Big(1- \frac{\beta_\ell}{1+\beta_\ell}\Big) \leq - \frac{\beta_\ell}{1+\beta_\ell}.
\]
Using this and $\beta_\ell \geq \underline{\beta}$ for all $\ell$, we obtain  
\[
\prod_{\ell=0}^{k-1} (1+\beta_\ell)^{-1} =\exp \Big(\sum_{\ell=0}^{k-1} \log (1+\beta_\ell)^{-1}\Big) 
\leq \exp  \Big(-\sum_{\ell=0}^{k-1} \frac{\beta_\ell}{1+\beta_\ell} \Big)
\leq \exp  \Big(-\frac{k \underline{\beta}}{1+\underline{\beta}}  \Big),
\]
which together with \eqref{rk-bound} and $r_0 \leq \delta$ implies that \eqref{thm2_linear} holds.

(ii) Suppose $\theta \in (1/2, 1)$.  
Clearly, \eqref{thm2_subli} holds if $r_k=0$. Now we assume that $r_k > 0$. It then follows from the monotonicity of $\{r_\ell\}$ that $r_\ell>0$ for all $0 \leq \ell < k$. Let $\psi(t) = \frac{1}{2\theta-1} t^{1-2\theta}$. Then we have
\beq \label{r-relation}
\psi(r_{\ell+1}) - \psi(r_\ell) = \int_{r_{\ell}}^{r_{\ell+1}} \psi'(t) dt = \int_{r_{\ell+1}}^{r_\ell} t^{-2\theta} dt \geq r_\ell^{-2\theta} (r_\ell - r_{\ell+1}) \qquad \forall 0 \leq \ell < k.
\eeq
For each $0 \leq \ell < k$, we consider two separate cases below.

Case a): $r_{\ell+1}^{-2\theta} \leq 2 r_\ell^{-2\theta}$. It along with \eqref{thm2_recur} and \eqref{r-relation} implies that
\[
\psi(r_{\ell+1}) - \psi(r_\ell) \geq \frac{1}{2} r_{\ell+1}^{-2\theta} (r_\ell - r_{\ell+1}) \overset{\eqref{thm2_recur}}{\geq} \frac{1}{2} \beta_\ell.
\]

Case b): $r_{\ell+1}^{-2\theta} > 2 r_\ell^{-2\theta}$. It leads to $r_{\ell+1}^{1-2\theta} >2^{\frac{2\theta-1}{2\theta}} r_\ell^{1 - 2\theta}$. By this, $r_\ell \leq \delta$, $\beta_\ell \geq \overline{\beta}$, and the expression of $\psi$, one has
\begin{align*}
\psi(r_{\ell+1}) - \psi(r_\ell) 
&= \frac{1}{2\theta-1} (r_{\ell+1}^{1-2\theta} - r_\ell^{1-2\theta}) > \frac{1}{2\theta-1} \left(2^{\frac{2\theta-1}{2\theta}} - 1\right) r_\ell^{1-2\theta} \\
&\geq \frac{1}{2\theta-1} \left(2^{\frac{2\theta-1}{2\theta}} - 1\right) \delta^{1-2\theta} \geq \frac{(2^{\frac{2\theta-1}{2\theta}} - 1) \delta^{1-2\theta}}{(2\theta-1)\, \overline{\beta}} \beta_\ell.
\end{align*}

Combining the above two cases, and using the definition of $C'$ in \eqref{thm2_const}, we obtain that $\psi(r_{\ell+1}) - \psi(r_\ell) \geq C' \beta_\ell$ for all $0 \leq \ell < k$. It then follows that 
\[
\psi(r_{k}) \geq \psi(r_0)+ C' \sum_{\ell=0}^{k-1} \beta_\ell  \geq C' \sum_{\ell=0}^{k-1} \beta_\ell.
\]
This and the expression of $\psi$ lead to
\[
r_{k} \leq \left( \frac{1}{C' (2\theta-1)} \right)^{\frac{1}{2\theta-1}} \left( \sum_{\ell=0}^{k-1} \beta_\ell \right)^{-\frac{1}{2\theta-1}} 
\leq \left( \frac{1}{C' (2\theta-1)\underline{\beta}} \right)^{\frac{1}{2\theta-1}} k^{-\frac{1}{2\theta-1}},
\]
and hence \eqref{thm2_subli} holds.
\end{proof}

We are now ready to prove Theorem~\ref{thm:subsolver_iters}.

\begin{proof}[\textbf{Proof of Theorem~\ref{thm:subsolver_iters}}] 
Suppose for contradiction that Algorithm~\ref{algo_implement:subsolver} runs for more than $\overline{K}_\theta$ outer iterations. Then there exists some $\ell \geq \overline{K}_\theta-1$ such that $\|z^{\ell+1} - z^\ell\| > \tau$. By \eqref{H1} and \eqref{H3}, one has
\beq \label{cor1_bound}
\|z^{\ell+1} - z^\ell\| \overset{\eqref{H1}}{\leq} \sqrt{\frac{r_\ell - r_{\ell+1}}{a_\ell}} \leq a_\ell^{-\frac{1}{2}} r_\ell^{\frac{1}{2}} \overset{\eqref{H3}}{\leq} (2\overline{\lambda})^{\frac{1}{2}} r_\ell^{\frac{1}{2}},
\eeq 
where $r_\ell = h(z^\ell) - h^*$. We next show that $r_\ell \leq \tau^2/(2\overline{\lambda})$ by considering two separate cases: $\theta=1/2$ and $\theta \in (1/2, 1)$.

Case (i): $\theta = 1/2$. By this, \eqref{cor1_K}, and $\ell \geq \overline{K}_\theta-1$, one has $\ell \geq \underline{\beta}^{-1} (1 + \underline{\beta}) \log (2\overline{\lambda} \delta \tau^{-2})$.  Using this relation and \eqref{thm2_linear}, we have $r_{\ell}  \leq \delta e^{ -\underline{\beta}(1 + \underline{\beta})^{-1}\ell} \leq \tau^2/(2\overline{\lambda})$.

Case (ii): $\theta \in (1/2, 1)$. Using this, \eqref{cor1_K}, and $\ell \geq \overline{K}_\theta-1$, we obtain that $\ell \geq \frac{1}{C'(2\theta-1)\underline{\beta}} \left( 2\overline{\lambda} \tau^{-2} \right)^{2\theta-1}$.  By this relation and \eqref{thm2_subli}, one has
\[
r_{\ell} \leq \left( \frac{1}{C' (2\theta-1)\underline{\beta}} \right)^{\frac{1}{2\theta-1}} \ell^{-\frac{1}{2\theta-1}} \leq \tau^2/(2\overline{\lambda}).
\]
We thus conclude that $r_\ell \leq \tau^2/(2\overline{\lambda})$. This together with \eqref{cor1_bound} implies $\|z^{\ell+1} - z^\ell\| \leq \tau$, which leads to a contradiction. Hence,  Algorithm~\ref{algo_implement:subsolver} runs at most $\overline{K}_\theta$ outer iterations.

We next show that \eqref{h-opt} holds. Notice from \eqref{lambda-bound} and \eqref{rk} that $\lambda_k \geq \underline{\lambda}$. By this and \eqref{lem8}, one has 
\beq \label{h-subdiff}
\operatorname{dist}(0, \partial h(z^{k+1})) \leq \big(L + \underline{\lambda}^{-1}\big) \|z^{k+1} - z^k\|.
\eeq
Since $h(z^0)-h^* \leq \delta$ and $\{h(z^\ell)\}$ is nonincreasing, it follows that $h(z^{k+1})-h^* \leq \delta$.  
Using this and \eqref{KL_subpr}, we have
\[
C (h(z^{k+1}) - h^*)^{\theta} \leq \operatorname{dist}(0, \partial h(z^{k+1})). 
\]
By this, \eqref{h-subdiff}, and $\|z^{k+1}-z^k\| \leq \tau$, one has
\[
h(z^{k+1}) - h^* \leq C^{-\frac{1}{\theta}}\big(\operatorname{dist}(0, \partial h(z^{k+1}))\big)^{\frac{1}{\theta}} 
\leq \big(C^{-1}(L + \underline{\lambda}^{-1})\big)^{\frac{1}{\theta}}\|z^{k+1} - z^k\|^{\frac{1}{\theta}} \leq \big(C^{-1}(L + \underline{\lambda}^{-1})\tau\big)^{\frac{1}{\theta}},
\]
and hence \eqref{h-opt} holds as desired.
\end{proof}

\subsection{Proof of the main results in Section~\ref{sec:FOD}}\label{sec:proof3}

In this subsection we prove Theorems \ref{thm:outer_bound} and \ref{thm:operations}. To proceed, we first establish several technical lemmas below.

\begin{lemma} \label{lem:subsolver_output}
Let $\gamma, \sigma, C, \theta, L_{\nabla f}, \epsilon, \cX^\rc_\epsilon$,\,and $\{\eta_\ell\}$ be given in \eqref{X-eps}, Assumption \ref{ass:Lip_Smo_KL}, and Algorithm \ref{algo_implement:whole}, respectively.  Suppose that $\{(x^\ell,y^\ell)\}^{k}_{\ell=0}$ are generated by Algorithm \ref{algo_implement:whole} for some $k \geq 1$ such that $x^{\ell} \in \cX^\rc_\epsilon$ for all $0\leq \ell < k$. Then, for all $0\leq \ell \leq k$,  it holds that
\begin{align}
& F^*(x^\ell) - F(x^\ell, y^\ell) \leq \min\big\{\gamma \epsilon^\sigma/2, \eta_{\ell} \big\}, \qquad \operatorname{dist}\big(y^\ell, Y^*(x^\ell)\big) \leq \frac{1}{C(1-\theta)} \min\big\{(\gamma/2)^{1-\theta} \epsilon^{\sigma (1-\theta)}, \eta_\ell^{1/2}\big\}, \label{lem9-ineq1-1}  \\ 
& \|\nabla^\rC_\cX F^*(x^\ell) -\nabla_x f(x^\ell, y^\ell)\| \leq \frac{L_{\nabla f}}{C(1-\theta)} \min\big\{(\gamma/2)^{1-\theta}\epsilon^{\sigma (1-\theta)},\eta_\ell^{1/2}\big\}. \label{lem9-ineq1-2} 
\end{align}
\end{lemma}

\begin{proof}
We first show that \eqref{lem9-ineq1-1} and \eqref{lem9-ineq1-2} hold for $\ell=0$. One can observe from Algorithm~\ref{algo_implement:whole} that the first relation in \eqref{lem9-ineq1-1} holds for $\ell=0$. We now show that the second relation in \eqref{lem9-ineq1-1} and \eqref{lem9-ineq1-2} also hold for $\ell=0$. By the assumption in this lemma, we know that $x^0\in \cX^\rc_\epsilon$, and hence $\dist(0, \partial \Psi(x^0)) > \epsilon$. Using this, $F^*(x^0) - F(x^0, y^0) \leq \min\{\gamma \epsilon^\sigma/2, 1\}$, and \eqref{Lev_set}, we see that $y^0 \in \cL(x^0)$ or $y^0 \in Y^*(x^0)$. In view of these, $x^0\in\cX$, and $\theta \in [1/2, 1)$, it follows from Lemma~\ref{lem:QG} that
\beq \label{lem9_dist_y0}
\dist(y^0, Y^*(x^0)) \leq \frac{1}{C(1-\theta)} (F^*(x^0) - F(x^{0}, y^{0}))^{1-\theta} \leq \frac{1}{C(1-\theta)} \min\big\{(\gamma/2)^{1-\theta} \epsilon^{\sigma (1-\theta)}, 1\big\}.
\eeq
This together with $\eta_0 = 1$ implies that the second relation in \eqref{lem9-ineq1-1} holds for $\ell = 0$. Moreover, one can see from $x^0 \in \cX$, $\dist(0, \partial \Psi(x^0)) > \epsilon$ and \eqref{U-eps} that $x^0\in\cU_\epsilon$. It then follows from Theorem \ref{thm:Phi_smooth} that $\nabla^\rC_\cX F^*(x^0)=\nabla_x f(x^0,y^*)$, where $y^*\in Y^*(x^0)$ with $\|y^*-y^0\|=\dist(y^0, Y^*(x^0))$. Using this, \eqref{lem9_dist_y0}, and the $L_{\nabla f}$-smoothness of $f$, one has
\begin{align*}
&\|\nabla^\rC_\cX F^*(x^0) - \nabla_x f(x^0,y^0)\| = \|\nabla_x f(x^0,y^*)-\nabla_x f(x^0,y^0)\| \leq L_{\nabla f}\|y^*-y^0\| \\
& =L_{\nabla f} \, \dist(y^0, Y^*(x^0)) \leq \frac{L_{\nabla f}}{C(1-\theta)} \min\big\{(\gamma/2)^{1-\theta}\epsilon^{\sigma(1-\theta)}, 1\big\}.
\end{align*}
This along with $\eta_0 = 1$ implies that \eqref{lem9-ineq1-2} also holds for $\ell = 0$.

We next show that \eqref{lem9-ineq1-1} and \eqref{lem9-ineq1-2} hold for $0<\ell \leq k$.  Notice from Algorithm \ref{algo_implement:whole} that $y^{\ell}$ is an approximate solution of the problem $\min_y \{-f(x^{\ell},y)+q(y)\}$ obtained by Algorithm~\ref{algo_implement:subsolver} with the initial point $y^{\ell-1}$, and the parameters $\overline{\lambda}, \rho, \tau$ specified in Algorithm~\ref{algo_implement:whole}. Then it follows from $\tau =\frac{C}{L_{\nabla f}+\underline{\lambda}^{-1}} \min\Big\{ (\tfrac{1}{2} \gamma \epsilon^\sigma)^{\theta}, \eta_{\ell}^{\frac{\theta}{2(1-\theta)}} \Big\}$, 
the definitions of $F^*$ and $F$, and Theorem \ref{thm:subsolver_iters} with $h(\cdot)=-f(x^{\ell},\cdot)+q(\cdot)$ that
\beq \label{opt-x-ell}
F^*(x^{\ell})-F(x^{\ell},y^{\ell}) \leq [C^{-1}(L _{\nabla f}+ \underline{\lambda}^{-1})\tau]^{\frac{1}{\theta}} =   \min \Big\{ \frac{1}{2} \gamma \epsilon^\sigma, \eta_{\ell}^{\frac{1}{2(1-\theta)}} \Big\}.
\eeq
This together with $\eta_\ell \in (0,1)$ and $\theta \in [1/2, 1)$ implies that the first relation in \eqref{lem9-ineq1-1} holds for $\ell>0$. In addition, notice from the assumption that $x^{\ell-1} \in \cX^\rc_\epsilon$. Also, observe from Algorithm~\ref{algo_implement:whole} that $r = \gamma \epsilon^\sigma / (4L_f)$ and $\|x^{\ell}-x^{\ell-1}\| \leq r$. It then follows that $\|x^{\ell}-x^{\ell-1}\| \leq \gamma \epsilon^\sigma / (4L_f)$,
which together with $x^{\ell-1} \in \cX^\rc_\epsilon$ implies that $\dist(0, \partial \Psi(x^\ell)) > \epsilon$. Using this and \eqref{opt-x-ell}, we obtain that $F^*(x^{\ell})-F(x^{\ell},y^{\ell}) \leq \gamma \dist(0,\partial \Psi(x^\ell))^\sigma$ and hence
$y^{\ell}\in \cL(x^\ell)$, where $\cL(\cdot)$ is defined in \eqref{Lev_set}. By this, \eqref{opt-x-ell}, $x^\ell \in \cX$, and Lemma \ref{lem:QG}, one can conclude that the second relation in \eqref{lem9-ineq1-1} holds for $\ell>0$.  Lastly, \eqref{lem9-ineq1-2} also holds for $\ell>0$, due to the second relation in \eqref{lem9-ineq1-1} and arguments similar to those used in the case $\ell=0$.  
\end{proof}

\begin{lemma} \label{lem:gen_converge}
Let $\epsilon>0$ be given, $M$, $\cX^\rc_\epsilon$ be defined in \eqref{Phi_smooth_constants} and \eqref{X-eps}, $L_f, L_{\nabla f}, C, \theta, \gamma, \sigma, \{\delta_\ell\}, \{\eta_\ell\}, \{L_\ell\}$ be given in Assumption \ref{ass:Lip_Smo_KL} and Algorithm~\ref{algo_implement:whole}, and let 
\begin{align}
& \Delta_k := 8 \Big[\Psi(x^0) - \Psi^* + \eta_{k+1} + \sum_{\ell=0}^{k}  \Big(1+\tfrac{L_{\nabla f}^2}{(1-\theta)^2 C^2 L_\ell} \Big)\eta_\ell + \sum_{\ell=0}^{k}\frac{\delta_\ell}{2} \Big],  \label{thm4_Deltak} \\
&\underline{K}_\epsilon := \max\{k \geq 1: \Delta_k/(kL_{\lceil k/2 \rceil}) \geq \gamma^2 \epsilon^{2\sigma} / (16L_f^2)\},  \label{K-eps-low}\\ 
& {\overline K}_\epsilon := \max \{k \geq 0: x^k \in \cX^\rc_\epsilon \}, \label{K-eps} \\
& \ell(k) := \mathop{\arg\min}_{\lceil k/2 \rceil \leq \ell \leq k} L_\ell \|x^{\ell+1} - x^\ell\|^2.  \label{thm4_def_hatk}
\end{align}
Let $\underline{K}_\epsilon < k \leq \overline{K}_\epsilon$ be given. Suppose that $\{(x^\ell,y^\ell)\}^{k}_{\ell=0}$ are generated by Algorithm \ref{algo_implement:whole} such that $x^\ell \in \cX^\rc_\epsilon$ for all 
$0 \leq \ell \leq k$. Then we have
\begin{equation} \label{eq:dist_to_stationary}
\operatorname{dist}\big(0, \partial \Psi(x^{\ell(k)+1})\big) \leq L_{\nabla f} \sqrt{ \frac{\Delta_k}{L_{\lceil k/2 \rceil} k} } + \sqrt{ \frac{L_{k} \Delta_k}{k} } + M \Big( \frac{\Delta_k}{L_{\lceil k/2 \rceil}k} \Big)^{\frac{\nu}{2}} + (1-\theta)^{-1} C^{-1} L_{\nabla f} \eta_{\lceil k/2 \rceil}^{\frac{1}{2}}.
\end{equation}
\end{lemma}

\begin{proof}
Notice from the above assumption that $\underline{K}_\epsilon < k \leq \overline{K}_\epsilon$ and $x^\ell \in \cX^\rc_\epsilon$ for all 
$0 \leq \ell \leq k$. We first show that for all $0 \leq \ell \leq k$, it holds that
\beq \label{thm4_descent}
F(x^{\ell+1}, y^{\ell+1}) + p(x^{\ell+1}) \leq F(x^\ell, y^\ell) + p(x^\ell) - \frac{L_\ell}{4} \|x^{\ell+1} - x^\ell\|^2 + \Big(1+\frac{L_{\nabla f}^2}{(1-\theta)^2 C^2 L_\ell}\Big)\eta_\ell + \frac{\delta_\ell}{2}.
\eeq
To this end, let us fix any $0 \leq \ell \leq k$. By optimality condition of \eqref{step:constrained_PG}, one has
\beq  \label{argmin_constrained}
\langle \nabla_x f(x^\ell, y^\ell), x^{\ell+1} \rangle + L_\ell \|x^{\ell+1} - x^\ell\|^2 + p(x^{\ell+1}) \leq \langle \nabla_x f(x^\ell, y^\ell), x^\ell \rangle + p(x^\ell).
\eeq
Observe from Algorithm~\ref{algo_implement:whole} that  $\|x^{\ell+1}-x^\ell\| \leq \gamma \epsilon^\sigma / (4L_f)$.
Using this relation, $x^\ell \in \cX^\rc_\epsilon$, and the definition of $ \cX^\rc_\epsilon$ in \eqref{X-eps},  we deduce that $\dist(0, \partial \Psi(x))>\epsilon$ for any $x\in [x^\ell, x^{\ell+1}]$. In addition, by $x^\ell, x^{\ell+1} \in\cX$, and the convexity of $\cX$, one can see that $[x^\ell, x^{\ell+1}]\subseteq \cX$. It follows from these and \eqref{U-eps} that $[x^\ell, x^{\ell+1}]\subseteq\cU_\epsilon$. In view of this, \eqref{sec4-Fstar-bound}, and the definition of $L_\ell$ in Algorithm~\ref{algo_implement:whole}, we have
\beq \label{Phi_kUpper}
F^*(x^{\ell+1}) \leq F^*(x^\ell) + \langle \nabla^\rC_\cX F^*(x^\ell), x^{\ell+1} - x^\ell \rangle + \frac{L_\ell}{2} \|x^{\ell+1} - x^\ell\|^2 + \frac{\delta_\ell}{2}.
\eeq
In addition, notice that $F(x^{\ell+1}, y^{\ell+1}) \leq F^*(x^{\ell+1})$. Using this, \eqref{lem9-ineq1-1}, \eqref{lem9-ineq1-2}, \eqref{argmin_constrained}, and \eqref{Phi_kUpper}, we obtain that
\begin{align}
&F(x^{\ell+1}, y^{\ell+1})+ p(x^{\ell+1}) \leq {F^*}(x^{\ell+1}) + p(x^{\ell+1}) \nn \\
&\overset{\eqref{Phi_kUpper}}{\leq} {F^*}(x^\ell) + \langle \nabla^\rC_\cX F^*(x^\ell), x^{\ell+1} - x^\ell \rangle + \frac{L_{\ell}}{2} \|x^{\ell+1} - x^\ell\|^2 + p(x^{\ell+1}) + \frac{\delta_\ell}{2} \nn \\
&= F(x^\ell, y^\ell) + \langle \nabla_x f(x^\ell, y^\ell), x^{\ell+1} - x^\ell \rangle + \frac{L_\ell}{2} \|x^{\ell+1} - x^\ell\|^2 + p(x^{\ell+1})  + F^*(x^\ell) - F(x^\ell, y^\ell) \nn \\
&\quad + \langle \nabla^\rC_\cX F^*(x^\ell) - \nabla_x f(x^\ell, y^\ell), x^{\ell+1} - x^\ell \rangle + \frac{\delta_\ell}{2} \nn \\
&\overset{\eqref{lem9-ineq1-1}\eqref{argmin_constrained}}{\leq} F(x^\ell, y^\ell) + p(x^\ell) - \frac{L_\ell}{2} \|x^{\ell+1} - x^\ell\|^2 + \eta_\ell + \langle \nabla^\rC_\cX F^*(x^\ell) - \nabla_x f(x^\ell, y^\ell), x^{\ell+1} - x^\ell \rangle + \frac{\delta_\ell}{2}\nn \\
&= F(x^\ell, y^\ell) + p(x^\ell) - \frac{L_{\ell}}{4} \|x^{\ell+1} - x^\ell\|^2 - \frac{L_{\ell}}{4} \|x^{\ell+1} - x^\ell\|^2 + \langle \nabla^\rC_\cX F^*(x^\ell) - \nabla_x f(x^\ell, y^\ell), x^{\ell+1} - x^\ell \rangle + \eta_\ell + \frac{\delta_\ell}{2} \nn \\
&\leq F(x^\ell, y^\ell) + p(x^\ell) - \frac{L_{\ell}}{4} \|x^{\ell+1} - x^\ell\|^2 + \frac{\|\nabla^\rC_\cX F^*(x^\ell) - \nabla_x f(x^\ell, y^\ell)\|^2}{L_{\ell}} + \eta_\ell + \frac{\delta_\ell}{2} \nn \\
&\overset{\eqref{lem9-ineq1-2}}{\leq} F(x^\ell, y^\ell) + p(x^\ell) - \frac{L_{\ell}}{4} \|x^{\ell+1} - x^\ell\|^2 + \Big(1+\frac{L_{\nabla f}^2}{(1-\theta)^2 C^2 L_{\ell}}\Big)\eta_\ell + \frac{\delta_\ell}{2}, \nn
\end{align}
where the fourth inequality follows from the Young's inequality $\langle u,v\rangle \le \alpha\|u\|^2/4 + \|v\|^2/\alpha$ for all $\alpha>0$ and $u,v\in\bR^n$. By this and the arbitrariness  of $\ell$, we see that \eqref{thm4_descent} holds for all $0 \leq \ell \leq k$.

Summing up \eqref{thm4_descent} over $\ell =0, \ldots, k $ yields
\beq \label{thm4_ineq1}
\begin{aligned}
\sum_{\ell=\lceil k/2 \rceil}^k L_\ell \|x^{\ell+1} - x^\ell\|^2 
&\leq 4 \Big[ F(x^0, y^0) + p(x^0) - F(x^{k+1}, y^{k+1}) - p(x^{k+1}) + \sum_{\ell=0}^{k} \Big(1+\tfrac{L_{\nabla f}^2}{(1-\theta)^2 C^2 L_\ell}\Big)\eta_\ell + \sum_{\ell=0}^{k} \frac{\delta_\ell}{2} \Big].
\end{aligned}
\eeq
For notational convenience, let $\hat{k}=\ell(k)$. By this, \eqref{thm4_Deltak}, and \eqref{thm4_def_hatk}, one has
\begin{align}
    &L_{\hat{k}} \|x^{\hat{k}+1} - x^{\hat{k}}\|^2 
    \overset{\eqref{thm4_def_hatk}}{\leq} \frac{1}{\lceil k/2 \rceil} \sum_{\ell=\lceil  k/2 \rceil}^k L_\ell \|x^{\ell+1} - x^\ell\|^2 \nn \\
    &\overset{\eqref{thm4_ineq1}}{\leq} \frac{4}{\lceil  k/2 \rceil} \Big[F(x^0, y^0) + p(x^0) - F(x^{k+1}, y^{k+1}) - p(x^{k+1}) + \sum_{\ell=0}^{k} \Big(1 + \tfrac{L_{\nabla f}^2}{(1-\theta)^2 C^2 L_\ell} \Big) \eta_\ell + \sum_{\ell=0}^{k} \frac{\delta_\ell}{2}\Big] \nn \\
    &\leq \frac{8}{k} \Big[ \Psi(x^0) - \Psi^* + \eta_{k+1} + \sum_{\ell=0}^{k} \Big(1 + \tfrac{L_{\nabla f}^2}{(1-\theta)^2 C^2 L_\ell} \Big) \eta_\ell + \sum_{\ell=0}^{k} \frac{\delta_\ell}{2} \Big] \overset{\eqref{thm4_Deltak}}{=} \frac{\Delta_k}{k}, \label{thm4_afterSum}
\end{align}
where the last inequality uses the fact that $F(x^0, y^0) \leq F^*(x^0)$, $\Psi(\cdot)=F^*(\cdot)+p(\cdot)$,  $\Psi(x^{k+1}) \geq \Psi^*$, and  $F^*(x^{k+1}) - F(x^{k+1}, y^{k+1}) \leq \eta_{k+1}$ due to \eqref{lem9-ineq1-1}. Additionally, by $\hat{k} \geq \lceil k/2 \rceil$, the monotonicity of $\{\delta_\ell\}$, $\nu \in (0, 1]$, and the definition of $L_\ell$ in Algorithm~\ref{algo_implement:whole}, one has $L_{\hat{k}} \geq L_{\lceil k/2 \rceil}$. It then together with \eqref{thm4_afterSum} implies that
\beq \label{thm4_hatk_between}
\|x^{\hat{k}+1} - x^{\hat{k}}\| \leq \sqrt{ \frac{\Delta_k}{L_{\hat{k}}k} } \leq \sqrt{ \frac{\Delta_k}{L_{\lceil k/2 \rceil}k} }.
\eeq
In addition, notice from Algorithm~\ref{algo_implement:whole} that $r = \gamma \epsilon^\sigma/(4L_f)$. By this, $k > \underline{K}_\epsilon$, and \eqref{K-eps-low}, one has $\Delta_{k}/(L_{\lceil k/2 \rceil}k)<r^2$. Using this and \eqref{thm4_hatk_between}, we have
\beq \label{xhat-diff}
\|x^{\hat{k}+1} - x^{\hat{k}}\| < r.
\eeq
This together with the first-order optimality condition of \eqref{step:constrained_PG} for
$x^{\hat{k}+1}$ implies that 
\beq \label{xhatk-opt}
0 \in \nabla_x f(x^{\hat{k}}, y^{\hat{k}}) + L_{\hat{k}}(x^{\hat{k}+1} - x^{\hat{k}}) + \partial p(x^{\hat{k}+1}).
\eeq

By $\hat k \leq k$ and the assumption that $x^\ell \in \cX^\rc_\epsilon$ for all $0 \leq \ell \leq k$, one has  $x^{\hat k} \in \cX^\rc_\epsilon$. Using this and \eqref{xhat-diff},  we conclude that $\dist(0, \partial \Psi(x^{\hat k}))>\epsilon$ and  $\dist(0, \partial \Psi(x^{\hat k+1}))>\epsilon$. By these and $x^{\hat k}, x^{\hat k+1} \in\cX$, one can see that $x^{\hat k}, x^{\hat k+1} \in \cU_\epsilon$. In view of this, \eqref{Phi_smooth_constants}, \eqref{xhat-diff}, and Theorem \ref{thm:Phi_smooth}, we obtain that 
\beq \label{grad-diff}
\| \nabla^\rC_\cX F^*(x^{\hat{k}+1}) - \nabla^\rC_\cX F^*(x^{\hat{k}}) \| \leq L_{\nabla f} \|x^{\hat{k}+1} - x^{\hat{k}}\| + M \|x^{\hat{k}+1} - x^{\hat{k}}\|^\nu.
\eeq
In addition, by $\lceil k/2 \rceil \leq \hat k \leq k$, the monotonicity of $\{\eta_\ell\}$ and $\{\delta_\ell\}$, and the definition of $L_\ell$, one has $\eta_{\hat k}\leq \eta_{\lceil k/2 \rceil}$ and $L_{\lceil k/2 \rceil} \leq L_{\hat{k}} \leq L_{k}$. Using these, \eqref{lem9-ineq1-2}, \eqref{thm4_hatk_between}, \eqref{xhatk-opt}, \eqref{grad-diff}, $\partial \Psi = \partial F^*+ \partial p$, and $\nabla^\rC_\cX F^*(\cdot) \in \partial F^*(\cdot)$, we have
\[
\begin{aligned}
\operatorname{dist}\big(0, \partial \Psi(x^{\hat{k}+1})\big) 
&\overset{\eqref{xhatk-opt}}{\leq} \|\nabla^\rC_\cX F^*(x^{\hat{k}+1}) - \nabla_x f(x^{\hat{k}}, y^{\hat{k}}) - L_{\hat{k}}(x^{\hat{k}+1} - x^{\hat{k}}) \| \\
&\leq \| \nabla^\rC_\cX F^*(x^{\hat{k}+1}) - \nabla^\rC_\cX F^*(x^{\hat{k}}) \| + \| \nabla^\rC_\cX F^*(x^{\hat{k}}) - \nabla_x f(x^{\hat{k}}, y^{\hat{k}}) \| + L_{\hat{k}} \|x^{\hat{k}+1} - x^{\hat{k}}\| \\
&\overset{\eqref{lem9-ineq1-2}\eqref{grad-diff}}{\leq} L_{\nabla f} \|x^{\hat{k}+1} - x^{\hat{k}}\| + M \|x^{\hat{k}+1} - x^{\hat{k}}\|^\nu + (1-\theta)^{-1} C^{-1} L_{\nabla f} \eta_{\hat{k}}^{\frac{1}{2}} + L_{\hat{k}} \|x^{\hat{k}+1} - x^{\hat{k}}\| \\
&\overset{\eqref{thm4_hatk_between}}{\leq} L_{\nabla f} \sqrt{ \frac{\Delta_k}{L_{\hat{k}}k} } + \sqrt{ \frac{L_{\hat{k}} \Delta_k}{k} } + M \Big( \frac{\Delta_k}{L_{\hat{k}}k} \Big)^{\frac{\nu}{2}} + (1-\theta)^{-1} C^{-1} L_{\nabla f} \eta_{\hat{k}}^{\frac{1}{2}} \\
&\leq L_{\nabla f} \sqrt{ \frac{\Delta_k}{L_{\lceil k/2 \rceil} k} } + \sqrt{ \frac{L_{k} \Delta_k}{k} } + M \Big( \frac{\Delta_k}{L_{\lceil k/2 \rceil}k} \Big)^{\frac{\nu}{2}} + (1-\theta)^{-1} C^{-1} L_{\nabla f} \eta_{\lceil k/2 \rceil}^{\frac{1}{2}}.
\end{aligned}
\]
This together with $\hat{k}=\ell(k)$ implies that the conclusion holds.
\end{proof}

The following lemma will be used to prove Theorem \ref{thm:outer_bound} subsequently. 

\begin{lemma}\label{lem:bound-extended}
Let $\zeta, a, b, \omega > 0$ be given. Then the following statements hold.
\begin{enumerate}[label=(\roman*)]
\item If $t \geq \left\lfloor 2\zeta^{-1} \log(1/\zeta) \right\rfloor_+ \!+\! 1$, then\, $t^{-1} \log t < \zeta$.

\item If $t \geq \max\Big\{ (2a\zeta^{-1})^{1/\omega},\, \Big( \left\lfloor 4b(\omega\zeta)^{-1} \log\left( 2b/(\omega\zeta) \right) \right\rfloor_+ \!+\! 1 \Big)^{1/\omega} \Big\}$, then\, $t^{-\omega}(a + b \log t) < \zeta$.
\end{enumerate}
\end{lemma}

\begin{proof}
We first prove statement (i). Fix any $t \geq \left\lfloor 2\zeta^{-1} \log(1/\zeta) \right\rfloor_+ \!+\! 1$. Let $ \phi(s) = s^{-1} \log s $. It can be verified that $\phi$ is strictly decreasing on $[e, \infty)$ and $\phi(s) \leq \phi(e) = 1/e$ for all $s>0$. The latter relation and $t>0$ imply that $t^{-1} \log t = \phi(t) < \zeta$ holds if $\zeta>1/e$. We now assume $\zeta \leq 1/e$. It then follows that $t > 2\zeta^{-1} \log(1/\zeta) \geq 2e$, which along with the strict monotonicity of $\phi$ on $[e, \infty)$ implies that
\[
t^{-1} \log t = \phi(t) < \phi(2\zeta^{-1} \log(1/\zeta)) = \frac{\zeta}{2} \frac{\log((2/\zeta) \log(1/\zeta))}{\log(1/\zeta)} = \frac{\zeta}{2} \Big(1 + \frac{\log(2\log(1/\zeta))}{\log(1/\zeta)}\Big).
\]
In addition, notice that $\zeta\log(1/\zeta)=\phi(1/\zeta)\leq 1/e<1/2$, which implies that $\log(2\log(1/\zeta)) \leq \log(1/\zeta)$. By this and the above inequality, one can conclude that statement (i) also holds if $\zeta \leq 1/e$.

We next prove statement (ii). Fix any $t \geq \max\{ (2a\zeta^{-1})^{1/\omega},\, ( \lfloor 4b(\omega\zeta)^{-1} \log\left( 2b/(\omega\zeta) ) \rfloor_+ \!+\! 1 \right)^{1/\omega} \}$. Since $t \geq (2a\zeta^{-1})^{1/\omega}$, we have $t^{-\omega} a \leq \zeta/2$. 
In addition, notice that $t^\omega \geq \lfloor 4b(\omega\zeta)^{-1} \log( 2b/(\omega\zeta) ) \rfloor_+ \!+\! 1 $, which together with statement (i) implies that 
$t^{-\omega} \log (t^\omega)<\omega\zeta/(2b)$. It then follows that $b t^{-\omega} \log t = b\omega^{-1} t^{-\omega} \log (t^\omega)  < \zeta/2$. By this and $t^{-\omega} a \leq \zeta/2$, one has $t^{-\omega}(a + b \log t) < \zeta$, and hence statement (ii) holds as desired.
\end{proof}

We are now ready to prove Theorems~\ref{thm:outer_bound} and \ref{thm:operations}.

\begin{proof}[\textbf{Proof of Theorem~\ref{thm:outer_bound}}] 
For notational convenience, let $k=\widehat{K}_\epsilon$. Suppose for contradiction that a $(\gamma \epsilon^\sigma/(4L_f),\epsilon)$-stationary point of problem \eqref{intro_problem} is not generated by Algorithm \ref{algo_implement:whole} in $k$ iterations.  

We first prove by induction that $\{(x^\ell,y^\ell)\}^k_{\ell=0}$ are successfully generated by Algorithm~\ref{algo_implement:whole}. Indeed, since $(x^0,y^0)$ is the initial point, it is generated by Algorithm~\ref{algo_implement:whole}. Now suppose $\{(x^i,y^i)\}^\ell_{i=0}$ are generated by Algorithm~\ref{algo_implement:whole} for some $0\leq \ell<k$. Since none of $\{x^i\}^\ell_{i=0}$ is a $(\gamma \epsilon^\sigma/(4L_f),\epsilon)$-stationary point of \eqref{intro_problem}, it follows that $x^i \in\cX^\rc_\epsilon$ for all $0 \leq i \leq \ell$, where $\cX^\rc_\epsilon$ is defined in \eqref{X-eps}. By this and Lemma~\ref{lem:subsolver_output}, one has that $F^*(x^{\ell}) - F(x^{\ell}, y^{\ell}) \le \gamma \epsilon^\sigma/2$. In addition, notice from \eqref{step:constrained_PG} that $x^{{\ell}+1}$ is well-defined and thus successfully generated by Algorithm~\ref{algo_implement:whole}, and moreover, $\|x^{{\ell}+1} - x^{\ell}\| \le r = \gamma \epsilon^\sigma / (4L_f)$. By these, Lemma~\ref{lem:Lips}, and the $L_f$-Lipschitz continuity of $F(\cdot, y)$ for each $y \in \mathcal{Y}$, we have
\begin{align}
F^*(x^{{\ell}+1}) - F(x^{{\ell}+1}, y^{\ell}) &= F^*(x^{{\ell}+1}) - F^*(x^{\ell}) + F^*(x^{\ell}) - F(x^{\ell}, y^{\ell}) + F(x^{\ell}, y^{\ell}) - F(x^{{\ell}+1}, y^{\ell}) \nn \\
&\leq 2L_f \|x^{{\ell}+1} - x^{\ell}\|+  \tfrac{1}{2} \gamma \epsilon^\sigma \leq 2L_f r + \tfrac{1}{2} \gamma \epsilon^\sigma = \gamma \epsilon^\sigma. \label{F-x-ell}
\end{align}
In addition, since $x^{\ell} \in \cX^\rc_\epsilon$,  $\|x^{{\ell}+1} - x^{\ell}\| \leq \gamma \epsilon^\sigma / (4L_f)$, and 
$x^{{\ell}+1}\in\cX$, one can see from \eqref{K-eps} that $\dist(0,\partial\Psi(x^{{\ell}+1}))>\epsilon$. This and \eqref{KL_y} imply that
\[
C (F^*(x^{{\ell}+1}) - F(x^{{\ell}+1}, y))^{\theta} \leq \operatorname{dist}(0, \partial_y F(x^{{\ell}+1}, y)) \quad \forall y \text{ with } F^*(x^{{\ell}+1}) > F(x^{{\ell}+1}, y) \geq F^*(x^{{\ell}+1}) - \gamma \epsilon^\sigma.
\]
Hence, \eqref{KL_subpr} holds for the function $h(\cdot)=-F(x^{{\ell}+1},\cdot)$ with $\delta=\gamma \epsilon^\sigma$.  It follows from this and \eqref{F-x-ell} that $y^{{\ell}}$ serves as a suitable initial point for applying Algorithm~\ref{algo_implement:subsolver} to solve the problem $\min_y -F(x^{{\ell}+1},y)$, or equivalently, $\min_y \{-f(x^{{\ell}+1},y)+q(y)\}$. In view of Theorem~\ref{thm:subsolver_iters}, $y^{{\ell}+1}$ is then successfully generated by Algorithm~\ref{algo_implement:whole} via applying Algorithm~\ref{algo_implement:subsolver} to this problem. Hence, the induction is completed.

We next derive a contradiction to the above hypothesis. By the definition of $\underline{L}$ in \eqref{thm5_A_underL}, $\nu \in (0, 1]$, $\delta_\ell \leq 1$, and the definition of $L_\ell$ in Algorithm~\ref{algo_implement:whole}, we see that $L_\ell \geq \underline{L}$ for all $0 \leq \ell \leq k$. In addition, observe from \eqref{thm5_C1}, \eqref{thm5_C2}, \eqref{K-eps-hat}, $\epsilon \in (0, 1/e]$, and $\nu \in (0, 1]$ that $\widehat{K}_\epsilon \geq 2$, which together with $k=\widehat{K}_\epsilon$ implies $k \geq 2$. In view of these, \eqref{thm5_A_underL}, \eqref{thm5_ab}, \eqref{thm4_Deltak}, and the definitions of $\{\delta_\ell\}$ and $\{\eta_\ell\}$, one has
\begin{align}
    \Delta_k &\overset{\eqref{thm4_Deltak}\eqref{thm5_A_underL}}{=} 8 \Big[\Psi(x^0) - \Psi^* + \eta_{k+1} + \sum_{\ell=0}^{k}  \Big(1+\frac{A}{L_\ell} \Big)\eta_\ell + \sum_{\ell=0}^{k}\frac{\delta_\ell}{2} \Big]
    \leq 8 \Big[\Psi(x^0) - \Psi^* + \sum_{\ell=0}^{k+1} \Big(1+\frac{A}{L_\ell} \Big)\eta_\ell + \sum_{\ell=0}^{k+1} \frac{\delta_\ell}{2} \Big] \nn \\
    &\leq 8 \Big[ \Psi(x^0) - \Psi^* + \Big( \frac{3}{2} + \frac{A}{\underline{L}} \Big) \sum_{\ell=0}^{k+1} \frac{1}{\ell+1} \Big] \leq 8 \left[ \Psi(x^0) - \Psi^* + \Big( \frac{3}{2} + \frac{A}{\underline{L}} \Big) \Big( 1 + \int_0^{k+1} \frac{1}{1+t} \, dt \Big) \right] \nn \\
    &= 8 \big[ \Psi(x^0) - \Psi^* + (3/2 + A\underline{L}^{-1}) \big( 1 + \log(k+2) \big) \big] \leq 8 \big[ \Psi(x^0) - \Psi^* + (3/2 + A\underline{L}^{-1}) ( 2 + \log k ) \big] \nn \\
    &= 8 \big[ \Psi(x^0) - \Psi^* + 3 + 2A\underline{L}^{-1} + (3/2 + A\underline{L}^{-1}) \log k \big] \overset{\eqref{thm5_ab}}{=} a + b \log k, \label{thm5_DeltaBd}
\end{align} 
where the last inequality follows from $\log(k+2) \leq \log k + 1$ due to $k \geq 2$. Similarly, one can show that $\Delta_{k'} \leq a + b\log k'$ for all $k' \geq k$. Let us fix any $k' \geq k$. Notice from $\delta_\ell = 1/(\ell+1)$ that $\delta_{\lceil k'/2 \rceil} \leq 2/k'$, which along with the definition of $L_\ell$ and $\nu \in (0, 1]$ implies that $L_{\lceil k'/2 \rceil} \geq (k'/2)^{(1-\nu)/(1+\nu)} M^{2/(1+\nu)}$. In view of these relations and $\nu \in (0, 1]$, we can see that
\beq \label{thm5_kover2Bound}
\frac{\Delta_{k'}}{L_{\lceil k'/2 \rceil} k'} \leq \frac{a + b \log k'}{(k'/2)^{\frac{1-\nu}{1+\nu}} k' M^{\frac{2}{1+\nu}}} \leq \frac{2(a + b \log k')}{k'^{\frac{2}{1+\nu}} M^{\frac{2}{1+\nu}}}.
\eeq
Using this and Lemma~\ref{lem:bound-extended}, we observe that $\Delta_{k'}/ (L_{\lceil k'/2 \rceil} k') < \gamma^2 \epsilon^{2\sigma}/(16L_f^2)$, since
\[
k' \geq k =\widehat{K}_\epsilon \geq \max\Big\{ \Big(\frac{64aL_f^2}{\gamma^2 \epsilon^{2\sigma}M^{2/(1+\nu)}}\Big)^{\frac{1+\nu}{2}}, \Big( \Big\lfloor \frac{64(1+\nu)bL_f^2}{\gamma^2 \epsilon^{2\sigma} M^{2/(1+\nu)}} \log \Big(\frac{32(1+\nu)bL_f^2}{\gamma^2 \epsilon^{2\sigma} M^{2/(1+\nu)}}\Big) \Big\rfloor_+ + 1 \Big)^{\frac{1+\nu}{2}} \Big\},
\]
where the last inequality is due to $\epsilon \in (0, 1/e]$, \eqref{thm5_C7}, \eqref{thm5_C8}, and \eqref{K-eps-hat}. It then follows from the arbitrariness  of $k'$ and the definition of $\underline{K}_\epsilon$ in \eqref{K-eps-low} that $\underline{K}_\epsilon < k$.

In addition, observe from $k=\widehat{K}_\epsilon$, $\epsilon \in (0, 1/e]$, \eqref{thm5_C5}, \eqref{thm5_C6}, and \eqref{K-eps-hat} that
\[
k \geq \max \Big\{ \Big(\frac{144aL_{\nabla f}^2}{\epsilon^2 M^{2/(1+\nu)}}\Big)^{\frac{1+\nu}{2}}, \; \Big( \Big\lfloor \frac{144(1+\nu)bL_{\nabla f}^2}{\epsilon^2 M^{2/(1+\nu)}} \log \Big(\frac{72(1+\nu)bL_{\nabla f}^2}{\epsilon^2 M^{2/(1+\nu)}}\Big) \Big\rfloor_+ + 1 \Big)^{\frac{1+\nu}{2}} \Big\}.
\]
It then follows from \eqref{thm5_kover2Bound} and Lemma~\ref{lem:bound-extended} that $\Delta_{k}/ (L_{\lceil k/2 \rceil} k) \leq \epsilon^2 / (36 L_{\nabla f}^2)$, which implies that
\beq \label{thm5_b1}
L_{\nabla f} \sqrt{ \frac{\Delta_k}{L_{\lceil k/2 \rceil} k} } \leq \frac{\epsilon}{6}.
\eeq
Similarly, notice from $k=\widehat{K}_\epsilon$, $\epsilon \in (0, 1/e]$, \eqref{thm5_C1}, \eqref{thm5_C3C4}, and \eqref{K-eps-hat} that
\[
k \geq \max \Big\{ \Big(\frac{36\underline{L}a}{\epsilon^2}\Big)^{\frac{1+\nu}{2\nu}}, \Big( \Big\lfloor \frac{36(1+\nu)b\underline{L}}{\nu \epsilon^2} \log \Big(\frac{18(1+\nu)b\underline{L}}{\nu \epsilon^2}\Big) \Big\rfloor_+ + 1 \Big)^{\frac{1+\nu}{2\nu}} \Big\}. 
\]
It then follows from \eqref{thm5_DeltaBd} and Lemma~\ref{lem:bound-extended} that $\Delta_{k}/ k^{2\nu/(1+\nu)} \leq \epsilon^2 / (18 \underline{L})$. By this, the definitions of $L_\ell$ and $\{\delta_\ell\}$, $\underline{L}=L_{\nabla f} + M^{2/(1+\nu)}$, $\nu \in (0, 1]$, and $\delta_\ell \leq 1$, we have
\begin{align*}
\frac{L_k \Delta_k}{k} &= \frac{(L_{\nabla f} + \delta_k^{\frac{\nu-1}{1+\nu}} M^{\frac{2}{1+\nu}}) \Delta_k}{k} \leq \frac{(L_{\nabla f} + M^{\frac{2}{1+\nu}}) \delta_k^{\frac{\nu-1}{1+\nu}} \Delta_k}{k} = \frac{\underline{L} (k+1)^{\frac{1-\nu}{1+\nu}} \Delta_k}{k} \leq \frac{2\underline{L}\Delta_k}{k^{\frac{2\nu}{1+\nu}}} \leq \frac{\epsilon^2}{9},
\end{align*}
where the second inequality follows from $(k+1)^{\frac{1-\nu}{1+\nu}} \leq 2k^{\frac{1-\nu}{1+\nu}}$ due to $k \geq 2$. Hence, we obtain
\beq \label{thm5_b2}
\sqrt{\frac{L_{k} \Delta_k}{k}} \leq \frac{\epsilon}{3}.
\eeq
Also, by $k=\widehat{K}_\epsilon$, $\epsilon \in (0, 1/e]$, \eqref{thm5_C2}, \eqref{thm5_C3C4}, and \eqref{K-eps-hat}, we can see that
\[
k \geq \max \Big\{ \Big(\frac{4a(3M)^{2/\nu}}{M^{2/(1+\nu)} \epsilon^{2/\nu}}\Big)^{\frac{1+\nu}{2}}, \Big( \Big\lfloor \frac{4b(1+\nu)(3M)^{2/\nu}}{M^{2/(1+\nu)} \epsilon^{2/\nu}} \log \Big(\frac{2b(1+\nu)(3M)^{2/\nu}}{M^{2/(1+\nu)} \epsilon^{2/\nu}}\Big) \Big\rfloor_+ + 1 \Big)^{\frac{1+\nu}{2}} \Big\}. 
\]
It then follows from \eqref{thm5_kover2Bound} and Lemma~\ref{lem:bound-extended} that $\Delta_{k}/ (L_{\lceil k/2 \rceil} k) \leq \epsilon^{2/\nu}/(3M)^{2/\nu}$, which implies that
\beq \label{thm5_b3}
M \Big( \frac{\Delta_k}{L_{\lceil k/2 \rceil}k} \Big)^{\frac{\nu}{2}} \leq \frac{\epsilon}{3}.
\eeq
Lastly, using $k=\widehat{K}_\epsilon$, \eqref{thm5_C3C4} and \eqref{K-eps-hat} yields $k \geq \lceil 72A \epsilon^{-2}\rceil$. By this and the definition of $\eta_k$, one has
\beq \label{thm5_b4}
A^{\frac{1}{2}} \eta_{\lceil k/2 \rceil}^{\frac{1}{2}} \leq \left( \frac{2A}{k+2} \right)^{\frac{1}{2}} \leq \frac{\epsilon}{6}.
\eeq

Recall from the above hypothesis that none of $\{x^\ell\}^k_{\ell=0}$ is a $(\gamma \epsilon^\sigma/(4L_f),\epsilon)$-stationary point of \eqref{intro_problem}. Hence, $x^\ell \in\cX^\rc_\epsilon$ for all $0 \leq \ell \leq k$. Moreover, it follows from this and the definition of $\overline{K}_\epsilon$ in~\eqref{K-eps} that $k \leq \overline{K}_\epsilon$. By $\underline{K}_\epsilon < k \leq \overline{K}_\epsilon$ and $x^\ell \in \cX^\rc_\epsilon$ for all $0 \leq \ell \leq k$, it follows from Lemma \ref{lem:gen_converge} that \eqref{eq:dist_to_stationary} holds for such $k$, which together with the definition of $A$ in \eqref{thm5_ab} leads to
\[
\operatorname{dist}\big(0, \partial \Psi(x^{\ell(k)+1})\big) \leq L_{\nabla f} \sqrt{ \frac{\Delta_k}{L_{\lceil k/2 \rceil} k} } + \sqrt{ \frac{L_{k} \Delta_k}{k} } + M \Big( \frac{\Delta_k}{L_{\lceil k/2 \rceil}k} \Big)^{\frac{\nu}{2}} + A^{\frac{1}{2}} \eta_{\lceil k/2 \rceil}^{\frac{1}{2}}.
\]
Combining this with \eqref{thm5_b1}, \eqref{thm5_b2}, \eqref{thm5_b3}, and \eqref{thm5_b4}, we obtain that
\[
\dist(0, \partial \Psi(x^{\hat{k}+1})) \leq \frac{\epsilon}{6} + \frac{\epsilon}{3} + \frac{\epsilon}{3} + \frac{\epsilon}{6} = \epsilon.
\]
In addition, notice from Algorithm~\ref{algo_implement:whole} that $\|x^{\hat{k}+1} - x^{\hat{k}}\| \leq  \gamma \epsilon^\sigma/(4L_f)$. It follows from these and the definition of $\cX^\rc_\epsilon$ in \eqref{X-eps} that $x^{\hat{k}} \notin \cX^\rc_\epsilon$. Since $\hat{k} \leq k$, this contradicts the assumption that $x^\ell \in \cX^\rc_\epsilon$ for all $0 \leq \ell \leq k$, which is implied by the hypothesis.  Hence, Algorithm \ref{algo_implement:whole} generates a pair $(x^k,y^k)$ such that $x^k$ is a $(\gamma \epsilon^\sigma/(4L_f),\epsilon)$-stationary point of problem \eqref{intro_problem} for some $0\leq k \leq \widehat{K}_\epsilon$. Moreover, it follows from Lemma \ref{lem:subsolver_output} that $y^k$ satisfies \eqref{yk-opt}.
\end{proof}

We next present the proof of Theorem~\ref{thm:operations}.

\begin{proof}[\textbf{Proof of Theorem~\ref{thm:operations}}] 
Let $\cX^\rc_\epsilon$ be defined in \eqref{K-eps}. The conclusion clearly holds if $x^0 \notin \cX^\rc_\epsilon$. Hence, we assume for the remainder of the proof that $x^0 \in \cX^\rc_\epsilon$. Given this and $\epsilon \in (0, 1/e]$, it follows from Theorem~\ref{thm:outer_bound} that there exists $0 \leq \overline{k} < \widehat{K}_\epsilon$ such that $x^\ell \in \cX^\rc_\epsilon$ for all  $0 \leq \ell \leq \overline{k}$ and $x^{\overline{k}+1} \notin \cX^\rc_\epsilon$. That is, the iterates $\{x^\ell\}^{\overline{k}}_{\ell=0}$ are not, but $x^{\overline{k}+1}$ is, a $(\gamma \epsilon^\sigma / (4L_f),\epsilon)$-stationary point of problem~\eqref{intro_problem}.

We first observe from Algorithm~\ref{algo_implement:whole} that the number of evaluations of the proximal operator $p$ equals the number of iterations. By this and $\overline{k} < \widehat{K}_\epsilon$, it follows that the total number of evaluations of $p$ to generate the $(\gamma \epsilon^\sigma / (4L_f),\epsilon)$-stationary point $x^{\overline{k}+1}$ is $\overline{k}+1 \leq \widehat{K}_\epsilon$.

We next show that the total number of evaluations of the proximal operator of $q$ performed in Algorithm~\ref{algo_implement:whole} to generate the $(\gamma \epsilon^\sigma / (4L_f),\epsilon)$-stationary point $x^{\overline{k}+1}$ is at most $\widehat{N}_2$. To this end, we analyze the number of evaluations of the proximal operator of $q$ conducted at each iteration $0 \leq \ell' \leq \overline{k}$, through its calls to Algorithm~\ref{algo_implement:subsolver}. As observed from Algorithm~\ref{algo_implement:whole} and the proof of Lemma~\ref{lem:gen_converge}, Algorithm~\ref{algo_implement:subsolver} is invoked to solve problem~\eqref{pr:KL_opt} with $h(\cdot) = -F(x^{\ell'+1}, \cdot)$, where $h$ satisfies condition~\eqref{KL_subpr} with $\delta = \gamma \epsilon^\sigma$.
By this, the definitions of $\tau$ and $\{\eta_\ell\}$ in Algorithm~\ref{algo_implement:whole}, $\overline{k} < \widehat{K}_\epsilon$, and \eqref{cor1_K}, it follows from Theorem~\ref{thm:subsolver_iters} that the number of outer iterations performed in Algorithm~\ref{algo_implement:subsolver} at each iteration $0 \leq \ell' \leq \overline{k}$ is at most $\overline{K}_{f, \theta}$, where $\overline{K}_{f, \theta}$ is defined in \eqref{cor1_Kover}. Using this and Theorem~\ref{thm:sub_LS}, we can see that at each iteration $0 \leq \ell' \leq \overline{k}$, the number of evaluations of $q$ is at most
\[
\Big( \Big\lceil \frac{\log (2L_{\nabla f} \overline{\lambda})}{\log \rho^{-1}} \Big\rceil_+ + 1 \Big) \overline{K}_{f, \theta}.
\]
By this bound, the fact that the total number of iterations performed by Algorithm~\ref{algo_implement:whole} to generate the $(\gamma \epsilon^\sigma / (4L_f),\epsilon)$-stationary point $x^{\overline{k}+1}$ is at most $\widehat{K}_\epsilon$, and \eqref{cor1_Neps}, we conclude that the total number of evaluations of the proximal operator of $q$ performed by Algorithm~\ref{algo_implement:whole} to generate a $(\gamma \epsilon^\sigma / (4L_f),\epsilon)$-stationary point is at most $\widehat{N}_\epsilon$.

Lastly, notice that the total number of evaluations of $\nabla f$ is no more than the sum of the total number of evaluations of the proximal operators of $ p $ and $ q $. It then follows that the total number of evaluations of $\nabla f$ performed in Algorithm~\ref{algo_implement:whole} to generate a $(\gamma \epsilon^\sigma / (4L_f),\epsilon)$-stationary point is at most $\widehat{K}_\epsilon + \widehat{N}_\epsilon$.
\end{proof}

\section{Concluding remarks} \label{sec:concluding}
In this paper, we consider a class of nonconvex--nonconcave minimax problems for which the inner maximization problem satisfies the local KL condition~\eqref{KL_y} with parameters $\theta \in [1/2,1)$, $\gamma > 0$, and $\sigma > 0$. Although minimax problems in which the inner maximization problem satisfies a global KL or PL condition form a special case of this class, the general complexity results developed here may become weak when applied to such problems. This is not surprising, as our complexity analysis targets worst-case instances in this broader class, where the inner maximization problem satisfies only a local KL condition over a region that shrinks as the outer variable approaches a stationary point.

Nevertheless, for the global KL or PL settings, as well as for other scenarios such as (i) $\theta \in (0,1/2)$, $\gamma > 0$, and $\sigma > 0$, and (ii) $\theta \in (0,1)$, $\gamma > 0$, and $\sigma = 0$, one can design tailored algorithms together with refined analyses that may yield stronger stationarity guarantees and improved complexity bounds.


\begin{thebibliography}{10}

\bibitem{arjovsky2017wasserstein}
M.~Arjovsky, S.~Chintala, and L.~Bottou.
\newblock {Wasserstein generative adversarial networks}.
\newblock In {\em International Conference on Machine Learning}, pages
  214--223, 2017.

\bibitem{attouch2013convergence}
H.~Attouch, J.~Bolte, and B.~F. Svaiter.
\newblock {Convergence of descent methods for semi-algebraic and tame problems:
  proximal algorithms, forward--backward splitting, and regularized
  Gauss--Seidel methods}.
\newblock {\em Mathematical programming}, 137(1):91--129, 2013.

\bibitem{bento2025convergence}
G.~Bento, B.~Mordukhovich, T.~Mota, and Y.~Nesterov.
\newblock {Convergence of descent optimization algorithms under
  Polyak-{\L}ojasiewicz-Kurdyka conditions}.
\newblock {\em Journal of Optimization Theory and Applications}, 207(3):41,
  2025.

\bibitem{bertsekas2009convex}
D.~Bertsekas.
\newblock {\em {Convex optimization theory}}, volume~1.
\newblock Athena Scientific, 2009.

\bibitem{bertsimas2011theory}
D.~Bertsimas, D.~B. Brown, and C.~Caramanis.
\newblock {Theory and applications of robust optimization}.
\newblock {\em SIAM review}, 53(3):464--501, 2011.

\bibitem{blanchet2025distributionally}
J.~Blanchet, J.~Li, S.~Lin, and X.~Zhang.
\newblock {Distributionally robust optimization and robust statistics}.
\newblock {\em Statistical Science}, 40(3):351--377, 2025.

\bibitem{bohm2022solving}
A.~B{\"o}hm.
\newblock {Solving nonconvex-nonconcave min-max problems exhibiting weak Minty
  solutions}.
\newblock {\em arXiv preprint arXiv:2201.12247}, 2022.

\bibitem{bolte2010characterizations}
J.~Bolte, A.~Daniilidis, O.~Ley, and L.~Mazet.
\newblock {Characterizations of {\L}ojasiewicz inequalities: subgradient flows,
  talweg, convexity}.
\newblock {\em Transactions of the American Mathematical Society},
  362(6):3319--3363, 2010.

\bibitem{cai2022accelerated}
Y.~Cai and W.~Zheng.
\newblock {Accelerated single-call methods for constrained min-max
  optimization}.
\newblock {\em arXiv preprint arXiv:2210.03096}, 2022.

\bibitem{clarke1975generalized}
F.~H. Clarke.
\newblock {Generalized gradients and applications}.
\newblock {\em Transactions of the American Mathematical Society},
  205:247--262, 1975.

\bibitem{dai2018sbeed}
B.~Dai, A.~Shaw, L.~Li, L.~Xiao, N.~He, Z.~Liu, J.~Chen, and L.~Song.
\newblock {SBEED: Convergent reinforcement learning with nonlinear function
  approximation}.
\newblock In {\em International Conference on Machine Learning}, pages
  1125--1134, 2018.

\bibitem{davis2019stochastic}
D.~Davis and D.~Drusvyatskiy.
\newblock Stochastic model-based minimization of weakly convex functions.
\newblock {\em SIAM Journal on Optimization}, 29(1):207--239, 2019.

\bibitem{drusvyatskiy2021nonsmooth}
D.~Drusvyatskiy, A.~D. Ioffe, and A.~S. Lewis.
\newblock {Nonsmooth optimization using Taylor-like models: error bounds,
  convergence, and termination criteria}.
\newblock {\em Mathematical Programming}, 185:357--383, 2021.

\bibitem{frankel2015splitting}
P.~Frankel, G.~Garrigos, and J.~Peypouquet.
\newblock {Splitting methods with variable metric for Kurdyka--{\L}ojasiewicz
  functions and general convergence rates}.
\newblock {\em Journal of Optimization Theory and Applications}, 165:874--900,
  2015.

\bibitem{goodfellow2020generative}
I.~Goodfellow, J.~Pouget-Abadie, M.~Mirza, B.~Xu, D.~Warde-Farley, S.~Ozair,
  A.~Courville, and Y.~Bengio.
\newblock {Generative adversarial networks}.
\newblock {\em Communications of the ACM}, 63(11):139--144, 2020.

\bibitem{huang2023enhanced}
F.~Huang.
\newblock {Enhanced adaptive gradient algorithms for nonconvex-PL minimax
  optimization}.
\newblock {\em arXiv preprint arXiv:2303.03984}, 2023.

\bibitem{jin2019minmax}
C.~Jin, P.~Netrapalli, and M.~I. Jordan.
\newblock {Minmax optimization: Stable limit points of gradient descent ascent
  are locally optimal}.
\newblock {\em arXiv preprint arXiv:1902.00618}, 2019.

\bibitem{karimi2016linear}
H.~Karimi, J.~Nutini, and M.~Schmidt.
\newblock {Linear convergence of gradient and proximal-gradient methods under
  the Polyak-{\L}ojasiewicz condition}.
\newblock In {\em Joint European conference on machine learning and knowledge
  discovery in databases}, pages 795--811, 2016.

\bibitem{li2018calculus}
G.~Li and T.~K. Pong.
\newblock {Calculus of the exponent of Kurdyka--{\L}ojasiewicz inequality and
  its applications to linear convergence of first-order methods}.
\newblock {\em Foundations of computational mathematics}, 18(5):1199--1232,
  2018.

\bibitem{li2025nonsmooth}
J.~Li, L.~Zhu, and A.~M.-C. So.
\newblock {Nonsmooth nonconvex--nonconcave minimax optimization: Primal--dual
  balancing and iteration complexity analysis}.
\newblock {\em Mathematical Programming}, pages 1--51, 2025.

\bibitem{liu2021first}
M.~Liu, H.~Rafique, Q.~Lin, and T.~Yang.
\newblock {First-order convergence theory for weakly-convex-weakly-concave
  min-max problems}.
\newblock {\em Journal of Machine Learning Research}, 22(169):1--34, 2021.

\bibitem{madry2017towards}
A.~Madry, A.~Makelov, L.~Schmidt, D.~Tsipras, and A.~Vladu.
\newblock {Towards deep learning models resistant to adversarial attacks}.
\newblock {\em arXiv preprint arXiv:1706.06083}, 2017.

\bibitem{nesterov2015universal}
Y.~Nesterov.
\newblock {Universal gradient methods for convex optimization problems}.
\newblock {\em Mathematical Programming}, 152(1):381--404, 2015.

\bibitem{nouiehed2019solving}
M.~Nouiehed, M.~Sanjabi, T.~Huang, J.~D. Lee, and M.~Razaviyayn.
\newblock {Solving a class of non-convex min-max games using iterative first
  order methods}.
\newblock {\em Advances in Neural Information Processing Systems}, 32, 2019.

\bibitem{omidshafiei2017deep}
S.~Omidshafiei, J.~Pazis, C.~Amato, J.~P. How, and J.~Vian.
\newblock {Deep decentralized multi-task multi-agent reinforcement learning
  under partial observability}.
\newblock In {\em International Conference on Machine Learning}, pages
  2681--2690, 2017.

\bibitem{pethick2023escaping}
T.~Pethick, P.~Latafat, P.~Patrinos, O.~Fercoq, and V.~Cevher.
\newblock {Escaping limit cycles: Global convergence for constrained
  nonconvex-nonconcave minimax problems}.
\newblock {\em arXiv preprint arXiv:2302.09831}, 2023.

\bibitem{rahimian2022frameworks}
H.~Rahimian and S.~Mehrotra.
\newblock {Frameworks and results in distributionally robust optimization}.
\newblock {\em Open Journal of Mathematical Optimization}, 3:1--85, 2022.

\bibitem{rockafellar1997convex}
R.~T. Rockafellar.
\newblock {\em {Convex analysis}}, volume~28.
\newblock Princeton university press, 1997.

\bibitem{rockafellar2009variational}
R.~T. Rockafellar and R.~J.-B. Wets.
\newblock {\em {Variational analysis}}, volume 317.
\newblock Springer Science \& Business Media, 2009.

\bibitem{sinha2017certifying}
A.~Sinha, H.~Namkoong, R.~Volpi, and J.~Duchi.
\newblock {Certifying some distributional robustness with principled
  adversarial training}.
\newblock {\em arXiv preprint arXiv:1710.10571}, 2017.

\bibitem{stein2009real}
E.~M. Stein and R.~Shakarchi.
\newblock {\em {Real analysis: measure theory, integration, and Hilbert
  spaces}}.
\newblock Princeton University Press, 2009.

\bibitem{xu2023zeroth}
Z.~Xu, Z.-Q. Wang, J.-L. Wang, and Y.-H. Dai.
\newblock {Zeroth-order alternating gradient descent ascent algorithms for a
  class of nonconvex-nonconcave minimax problems}.
\newblock {\em Journal of Machine Learning Research}, 24(313):1--25, 2023.

\bibitem{yang2022faster}
J.~Yang, A.~Orvieto, A.~Lucchi, and N.~He.
\newblock {Faster single-loop algorithms for minimax optimization without
  strong concavity}.
\newblock In {\em International Conference on Artificial Intelligence and
  Statistics}, pages 5485--5517, 2022.

\bibitem{zheng2025doubly}
T.~Zheng, A.~M.-C. So, and J.~Li.
\newblock {Doubly smoothed optimistic gradients: A universal approach for
  smooth minimax problems}.
\newblock {\em arXiv preprint arXiv:2506.07397}, 2025.

\bibitem{zheng2023universal}
T.~Zheng, L.~Zhu, A.~M.-C. So, J.~Blanchet, and J.~Li.
\newblock {Universal gradient descent ascent method for nonconvex-nonconcave
  minimax optimization}.
\newblock {\em Advances in Neural Information Processing Systems},
  36:54075--54110, 2023.

\end{thebibliography}
\end{document}